\documentclass[11pt]{article}
\usepackage[english]{babel}
\usepackage[utf8]{inputenc}
\usepackage[T1]{fontenc}
\usepackage{amsfonts}
\usepackage{amsmath}
\usepackage{amssymb}
\usepackage{amsthm}
\usepackage{bm}
\usepackage{upgreek}

\newcommand{\bc}{\bm{c}}
\newcommand{\bN}{\bm{N}}
\newcommand{\ga}{\mathfrak{a}}
\newcommand{\gA}{\mathfrak{A}}
\newcommand{\gb}{\mathfrak{b}}
\newcommand{\gc}{\mathfrak{c}}

\newcommand{\gC}{\mathfrak{C}}

\newcommand{\C}{\mathbb{C}}
\newcommand{\N}{\mathbb{N}}

\newcommand{\R}{\mathbb{R}}

\newcommand{\boC}{\mathcal{C}}
\newcommand{\boE}{\mathcal{E}}

\newcommand{\boG}{\mathcal{G}}
\newcommand{\boH}{\mathcal{H}}

\newcommand{\boL}{\mathcal{L}}
\newcommand{\boM}{\mathcal{M}}
\newcommand{\boN}{\mathcal{N}}
\newcommand{\boO}{\mathcal{O}}

\newcommand{\boR}{\mathcal{R}}

\newcommand{\boU}{\mathcal{U}}
\newcommand{\boV}{\mathcal{V}}

\newcommand{\eps}{\varepsilon}
\newcommand{\ch}{{\rm ch}}
\newcommand{\Adm}{{\rm Adm}}

\newcommand{\Pos}{{\rm Pos}}

\newcommand{\sign}{{\rm sign}}
\renewcommand{\th}{{\rm th}}

\newtheorem{cor}{Corollary}
\newtheorem{lemma}{Lemma}
\newtheorem{prop}{Proposition}
\newtheorem{step}{Step}
\newtheorem{theorem}{Theorem}
\newtheorem*{theorem*}{Theorem}
\theoremstyle{definition}
\newtheorem*{merci}{Acknowledgments}
\newtheorem{remark}{Remark}

\begin{document}

\title{Stability in the energy space for chains of solitons of the one-dimensional Gross-Pitaevskii equation}
\author{
\renewcommand{\thefootnote}{\arabic{footnote}}
Fabrice B\'ethuel \footnotemark[1], Philippe Gravejat \footnotemark[2], Didier Smets \footnotemark[3]}
\footnotetext[1]{Laboratoire Jacques-Louis Lions, Universit\'e Pierre et Marie Curie, Bo\^ite Courrier 187, 75252 Paris Cedex 05, France. E-mail: bethuel@ann.jussieu.fr}
\footnotetext[2]{Centre de Math\'ematiques Laurent Schwartz, \'Ecole Polytechnique, 91128 Palaiseau Cedex, France. E-mail: gravejat@math.polytechnique.fr}
\footnotetext[3]{Laboratoire Jacques-Louis Lions, Universit\'e Pierre et Marie Curie, Bo\^ite Courrier 187, 75252 Paris Cedex 05, France. E-mail: smets@ann.jussieu.fr}
\maketitle

\begin{abstract}
We establish the stability in the energy space for sums of solitons of the one-dimensional Gross-Pitaevskii equation when their speeds are mutually distinct and distinct from zero, and when the solitons are initially well-separated and spatially ordered according to their speeds.
\end{abstract}

\section{Introduction}

In this paper, we focus on the one-dimensional Gross-Pitaevskii equation 
\renewcommand{\theequation}{GP}
\begin{equation}
\label{GP}
i \partial_t \Psi + \partial_x^2 \Psi + \Psi \big( 1 - |\Psi|^2 \big) = 0,
\end{equation}
for a function $\Psi: \R \times \R \to \C$, supplemented with the boundary condition at infinity
\renewcommand{\theequation}{\arabic{equation}}
\setcounter{equation}{0}
\begin{equation}
\label{bdinfini}
|\Psi(x, t)| \to 1, \quad {\rm as} \quad |x| \to + \infty.
\end{equation}
The Gross-Pitaevskii equation is a defocusing nonlinear Schr\"odinger equation. In one spatial dimension, solutions with finite Ginzburg-Landau energy
$$\boE(\Psi) := \frac{1}{2} \int_\R |\partial_x \Psi|^2 + \frac{1}{4} \int_\R (1 - |\Psi|^2)^2,$$
are globally defined, and the energy is conserved along the flow.

In 1972, V.E. Zakharov and A.B. Shabat \cite{ShabZak2} have shown that the Gross-Pitaevskii equation supplemented with condition \eqref{bdinfini} is integrable by means of an inverse scattering transform. As a consequence, they exhibited exact soliton and multi-soliton solutions. It is generally believed that such solutions play a central role in the long-time asymptotics of the flow, and sometimes conjectured that any initial datum with finite Ginzburg-Landau energy should eventually resolve into a multi-soliton solution plus a vanishing (in an appropriate sense) dispersive tail. Detailed asymptotic expansions for smooth and fast decaying solutions, based on their scattering data at the initial time, have actually been obtained by A.H. Vartanian in \cite{Vartani2}, using reductions to Riemann-Hilbert problems. To our knowledge, for initial data in the energy space, there is no rigorous evidence supporting the aforementioned behaviour further than the orbital stability of a single soliton (see e.g. \cite{LinZhiw1, BetGrSa2, BeGrSaS1, GeraZha1}). The main goal of this paper is to provide a partial justification through the proof of the orbital stability in the energy space of some finite sums of solitons along the Gross-Pitaevskii flow. Our approach does not make any use of the integrability of the equation, and is largely influenced by the corresponding study of the generalized Korteweg-de Vries equation (see \cite{MarMeTs1}) and of the nonlinear Schr\"odinger equation with vanishing data at infinity (see \cite{MarMeTs2}) by Y. Martel, F. Merle and T.-P. Tsai. The main ingredient is a monotonicity formula for the momentum (see Subsection \ref{sub:monotonicity}).

\subsection{Soliton solutions to the Gross-Pitaevskii equation}
\label{sub:soli}

A soliton with speed $c$ is a solution to \eqref{GP} which takes the form
$$\Psi(x, t) := U_c(x - c t).$$
Its profile $U_c$ is a solution to the ordinary differential equation
\begin{equation}
\label{eq:solc}
- i c \partial_x U_c + \partial_{xx}^2 U_c + U_c \big( 1 - |U_c|^2 \big) = 0.
\end{equation}
The solutions to \eqref{eq:solc} with finite energy are classified and indeed explicit. For $|c| \geq \sqrt{2}$, all of them are identically constant. When $|c| < \sqrt{2}$, there exist non-constant solutions with finite energy. Up to the invariances of the problem, i.e. the multiplication by a constant of modulus one and the invariance by translation, they are uniquely given by the expression
\begin{equation}
\label{form:solc}
U_c(x) = \sqrt{\frac{2 - c^2}{2}} \th \Big( \frac{\sqrt{2 - c^2} x}{2} \Big) + i \frac{c}{\sqrt{2}}.
\end{equation}
Notice that the soliton $U_c$ does not vanish on $\R$ when $c \neq 0$. As a matter of fact, our stability analysis below requires non-vanishing solutions. Indeed, it is performed on a reformulation of \eqref{GP} which only makes sense for such solutions, and which we introduce next.

\subsection{Hydrodynamical form of the Gross-Pitaevskii equation}

Provided a solution $\Psi$ to \eqref{GP} does not vanish, it may be written, at least formally, as 
$$\Psi := \varrho \exp i \varphi,$$
where $\varrho := |\Psi|$. Still on a formal level, the functions $\eta := 1 - \varrho^2$ and $v := \partial_x \varphi$ are then solutions to an hydrodynamical form of \eqref{GP}, namely
\renewcommand{\theequation}{HGP}
\begin{equation}
\label{HGP}
\left\{ \begin{array}{ll}
\partial_t \eta = \partial_x \big( 2v - 2 \eta v \big),\\[5pt]
\displaystyle \partial_t v = \partial_x \Big( \eta - v^2 - \partial_x \Big( \frac{\partial_x \eta}{2 (1 - \eta)} \Big) + \frac{(\partial_x \eta)^2}{4 (1 - \eta)^2} \Big).
\end{array} \right.
\end{equation}
For $k \geq 0$, we introduce the Hilbert spaces $X^k(\R) := H^{k + 1}(\R) \times H^k(\R)$, equipped with the norm 
$$\| (\eta, v) \|_{X^k}^2 := \| \eta \|_{H^{k + 1}}^2 + \| v \|_{H^k}^2,$$
and their open subsets
\footnote{The notation $\boN\boV$ stands for non-vanishing.}
$$\boN\boV^k(\R) := \big\{ (\eta, v) \in X^k(\R), \ {\rm s.t.} \ \max_{x \in \R} \eta(x) < 1 \big\},$$
consisting of states with non-vanishing corresponding $\Psi$. In the sequel, we set $X(\R) := X^0(\R)$ and $\boN\boV(\R) := \boN\boV^0(\R)$.

The Ginzburg-Landau energy
\footnote{When $\Psi$ may be written as $\Psi = \varrho \exp i \varphi$, the quantities $\boE(\Psi)$ and $E(\eta, v)$, with $\eta = 1 - \varrho^2$ and $v = \partial_x \varphi$, are exactly equal.}
\renewcommand{\theequation}{\arabic{equation}}
\setcounter{equation}{3}
\begin{equation}
\label{def:E}
E(\eta, v) := \int_\R e(\eta, v) := \frac{1}{8} \int_\R \frac{(\partial_x \eta)^2}{1 - \eta} + \frac{1}{2} \int_\R (1 - \eta) v^2 + \frac{1}{4} \int_\R \eta^2,
\end{equation}
and the momentum
$$P(\eta, v) := \int_\R p(\eta, v) := \frac{1}{2} \int_\R \eta v,$$
are well-defined and smooth functionals on $\boN\boV(\R)$. In view of \eqref{def:E}, the space $\boN\boV(\R)$ is the energy space for states $(\eta, v)$ satisfying the non-vanishing condition $\max_{x \in \R} \eta(x) < 1$.

Concerning the Cauchy theory for \eqref{HGP}, we have 

\begin{theorem}[\cite{Tartous0}]
\label{thm:cauchyhgp}
Let $k \in \N$ and $(\eta^0, v^0) \in \boN\boV^k(\R)$. There exists a maximal time $T_{\max} > 0$ and a unique solution $(\eta, v) \in \boC^0([0, T_{\max}), \boN\boV^k(\R))$ to equation \eqref{HGP} with initial datum $(\eta^0, v^0)$. The maximal time $T_{\max}$ is continuous with respect to the initial datum $(\eta^0, v^0)$, and is characterized by 
$$\lim_{t \to T_{\max}} \max_{x \in \R} \eta(x) = 1 \quad {\rm if} \quad T_{\rm max} < + \infty.$$ 
The flow map $(\eta^0, v^0) \mapsto (\eta, v)$ is locally well-defined and continuous from $\boN\boV^k(\R)$ to $\boC^0([0, T], \linebreak[1] \boN\boV^k(\R))$ for any $T < T_{\max}$. Moreover, the energy $E$ and the momentum $P$ are constant along the flow. 
\end{theorem}

We refer to \cite{Tartous0} for a proof. It is a consequence of the Cauchy theory for \eqref{GP} in Zhidkov spaces $Z^k(\R)$ (see e.g. \cite{Zhidkov1, Gallo3, Gerard1}), and the fact that the mapping $\Psi \mapsto (\eta, v)$ is locally Lipschitz from $Z^k(\R)$ to $X^k(\R)$ provided that $(\eta, v) \in \boN\boV(\R)$. The uniqueness part requires slightly more care. 

\subsection{Statement of the main result}

When $c \neq 0$, the solitons $U_c$ described in Subsection \ref{sub:soli} do not vanish and may thus be written under the form
$$U_c := \varrho_c \exp i \varphi_c.$$
In view of formula \eqref{form:solc}, the maps $\eta_c := 1 - \varrho_c^2$ and $v_c := \partial_x \varphi_c$ are given by
\footnote{A striking point in formula \eqref{form:etavc} is the fact that, up to some scaling, the map $\eta_c$ is exactly the soliton of the Korteweg-de Vries equation
\renewcommand{\theequation}{KdV}
\begin{equation}
\label{KdV}
\partial_t u + \partial_x^3 u + u \partial_x u = 0.
\end{equation}
Like the Gross-Pitaevskii equation, the Korteweg-de Vries equation indeed owns a family of solitons $u_\sigma$, given, up to translations, by
$$u_\sigma(x) := \frac{3 \sigma}{\ch(\frac{\sqrt{\sigma} x}{2})^2},$$
for some positive speed $\sigma$. Actually, the link between the two equations does not reduce to the level of the solitons. Both of them are integrable by means of the inverse scattering transform (see e.g. \cite{Miura2} for the Korteweg-de Vries equation), and moreover, as was pointed out in \cite{BeGrSaS1, BeGrSaS2} (see also \cite{ChirRou2}), the Korteweg-de Vries equation appears as the limit of system \eqref{HGP} in some long-wave regime.}
\renewcommand{\theequation}{\arabic{equation}}
\setcounter{equation}{4}
\begin{equation}
\label{form:etavc}
\eta_c(x) = \frac{2 - c^2}{2 \ch \big( \frac{\sqrt{2 - c^2} x}{2} \big)^2} \quad {\rm and} \quad v_c(x) = - \frac{c \eta_c(x)}{2 \big( 1 - \eta_c(x) \big)} = - \frac{c (2 - c^2)}{2 \big( 2 \ch \big( \frac{\sqrt{2 - c^2} x}{2} \big)^2 - 2 + c^2 \big)}.
\end{equation}
In the sequel, we set
$$Q_{c, a} := \big( \eta_{c, a}, v_{c, a} \big) := \big( \eta_c(\cdot - a), v_c(\cdot - a) \big),$$
for $0 < |c| < \sqrt{2}$ and $a \in \R$. We introduce the set of admissible speeds as
$$\Adm_N := \big\{ \gc := (c_1, \ldots, c_N) \in (- \sqrt{2}, \sqrt{2})^N, \ {\rm s.t.} \ c_k \neq 0, \ \forall 1 \leq k \leq N \big\},$$
and, for $\gc \in \Adm_N$ and $\ga := (a_1, \ldots, a_N) \in \R^N$, we define
$$R_{\gc, \ga} := \big( \eta_{\gc, \ga}, v_{\gc, \ga} \big) = \sum_{k = 1}^N Q_{c_k, a_k} = \bigg( \sum_{k = 1}^N \eta_{c_k, a_k}, \sum_{k = 1}^N v_{c_k, a_k} \bigg) = \bigg( \sum_{k = 1}^N \eta_{c_k}(\cdot - a_k), \sum_{k = 1}^N v_{c_k}(\cdot - a_k) \bigg).$$
In the sequel, we refer to small perturbations of the functions $R_{\gc, \ga}$ as chains of solitons. We also denote
$$\nu_\gc := \min \Big\{ \sqrt{2 - c_k^2}, \ 1 \leq k \leq N \Big\}.$$ 
Our main result is

\begin{theorem}
\label{thm:orbistab}
Let $ \gc^* \in \Adm_N$ be such that 
\begin{equation}
\label{cond:speed}
c_1^* < \ldots < c_N^*.
\end{equation}
There exist positive numbers $\alpha^*$, $L^*$ and $A^*$, depending
\footnote{A close inspection of the proof actually shows that $\alpha^*$, $L^*$ and $A^*$ only depend on a lower bound on $\nu_{\gc^*}$, on $\min \{ |c_j^* - c_k^*|, \ 1 \leq j < k \leq N \}$ and on $\min \{ |c_k^*|, \ 1 \leq k \leq N \}$.}
only on $\gc^*$, such that the following holds. Assume that $(\eta^0, v^0) \in \boN\boV(\R)$ satisfies 
\begin{equation}
\label{cond:init}
\alpha^0 := \Big\| (\eta^0, v^0) - R_{\gc^*, \ga^0} \Big\|_X \leq \alpha^*,
\end{equation}
for some points $\ga^0 = (a_1^0, \ldots, a_N^0) \in \R^N$ such that
\begin{equation}
\label{cond:separ}
L^0 := \min \big\{ a_{k + 1}^0 - a_k^0, \ 1 \leq k \leq N - 1 \big\} \geq L^*.
\end{equation}
Then the unique solution $(\eta, v)$ to \eqref{HGP} with initial datum $(\eta^0, v^0)$ is globally defined on $[0, + \infty)$, and there exists a function $\ga = (a_1, \ldots, a_N) \in \boC^1([0, + \infty), \R^N)$ such that 
\begin{equation}
\label{eq:modul}
\sum_{k = 1}^N \big| a_k'(t) - c_k^* \big| \leq A^* \Big( \alpha^0 + \exp \Big( - \frac{\nu_{\gc^*} L^0}{33} \Big) \Big),
\end{equation}
and
\begin{equation}
\label{orbistab}
\Big\| (\eta(\cdot, t), v(\cdot, t)) - R_{\gc^*, \ga(t)} \Big\|_X \leq A^* \Big( \alpha^0 + \exp \Big( - \frac{\nu_{\gc^*} L^0}{33} \Big) \Big),
\end{equation}
for any $t \in [0, + \infty)$.
\end{theorem}

In different terms, chains of solitons with strictly ordered speeds are preserved along the Gross-Pitaevskii flow in positive time. Since the initial speeds are strictly ordered, the mutual distances between the individual solitons appearing in the chain increase proportionally with time, so that the interactions between solitons decrease to zero. As a matter of fact, integrating \eqref{eq:modul} yields
$$a_{k + 1}(t) - a_{k}(t) \geq \Big( c_{k + 1}^* - c_k^* - 2 A^* \Big( \alpha^0 + \exp \Big( - \frac{\nu_{\gc^*} L^0}{33} \Big) \Big) \Big) t + a_{k + 1}(0) - a_k(0).$$
A typical example of chain of solitons is given by the exact multi-solitons constructed by Zakharov and Shabat \cite{ShabZak2} (see also \cite{FaddTak0}). In particular, Theorem \ref{thm:orbistab} provides the orbital stability of multi-solitons in the energy space in positive time.

In view of the time reversibility of the Gross-Pitaevskii equation, Theorem \ref{thm:orbistab} also provides the stability in negative time of chains of solitons with reversely ordered speeds. To our knowledge, the question of the stability in the energy space of chains of solitons for both positive and negative times remains open when the chain includes at least two solitons. For the generalized Korteweg-de Vries equations, results of this type have recently been obtained by Martel and Merle \cite{MartMer5, MartMer6}.

Finally, we believe that our arguments still apply to non-integrable versions of the nonlinear Schr\"o\-dinger equation with non-vanishing condition at infinity, provided individual solitons are themselves orbitally stable. The latter condition has recently been treated in details in a work by Chiron \cite{Chiron7}, where numerous examples are presented.

\subsection{Elements in the proofs}

We will present a detailed description of the elements in the proof of Theorem \ref{thm:orbistab} in Subsections \ref{sub:coer} to \ref{sub:end} below. We precede this description by a voluntarily oversimplified overview in the case of a two-soliton, in the hope that this may ease the reading of the detailed description.

The solitons $Q_c$ may be characterized as minimizers of some variational problems. In terms of $Q_c = (\eta_c, v_c)$, equation \eqref{eq:solc} reduces to
\begin{equation}
\label{eq:crit}
E'(Q_c) + c P'(Q_c) = 0.
\end{equation}
Equation \eqref{eq:crit} is the Euler-Lagrange equation for the minimization of the Ginzburg-Landau energy $E$ under fixed scalar momentum $P$, and the speed $c$ is the corresponding Lagrange multiplier. As a matter of fact, we have
\begin{equation}
\label{eq:minimizingP}
E(Q_c) = \min \Big\{ E(\eta, v), \ (\eta, v) \in X(\R) \ {\rm s.t.} \ \eta < 1 \ {\rm and} \ P(\eta, v) = P(Q_c) \Big\},
\end{equation}
for any $c \in (- \sqrt{2}, \sqrt{2}) \setminus \{ 0 \}$ (see e.g. \cite{BetGrSa2}), and $Q_c$ is the only minimizer up to the invariances of the problem. This minimizing property, combined with the conservation of $E$ and $P$ by the Gross-Pitaevskii flow, yields the orbital stability of single solitons. Instability of a {\it chain} of solitons, should it happen, could thus only arise from their mutual interactions. Assume for simplicity that at some fixed time the solution at hand is a chain of two solitons
$$(\eta, v) \simeq Q_{c_1, a_1} + Q_{c_2, a_2},$$
where $a_2 - a_1 \gg 1$. In particular, the total energy $E$ and momentum $P$ satisfy
$$E \simeq E(Q_{c_1}) + E(Q_{c_2}) \quad {\rm and} \quad P \simeq P(Q_{c_1}) + P(Q_{c_2}).$$
Assume next that at some later time, interaction has occurred and 
$$(\eta, v) \simeq B_1(\cdot - b_1) + B_2(\cdot - b_2),$$
where still $b_2 - b_1 \gg 1$, but $B_1$ and $B_2$ need no longer be solitons. Set
$$\delta P := P(Q_{c_1}) - P(B_1).$$
Since the total momentum is an invariant of the flow, we also have
$$\delta P \simeq P(B_2) - P(Q_{c_2}),$$
so that $\delta P$ may be understood as the amount of momentum transfered from $Q_{c_1, a_1}$ towards $Q_{c_2, a_2}$ (we let aside here the possibility of transfer to the background).

From \eqref{eq:minimizingP}, we infer that the energies $E(B_1)$ and $E(B_2)$ are necessarily larger than the energies of the corresponding solitons with momentum $P(B_1)$, respectively $P(B_2)$. Since, in particular from \eqref{eq:crit},
$$\frac{d}{dc} \Big( E(Q_c) \Big) = - c \frac{d}{dc} \Big( P(Q_c) \Big),$$
we obtain at first order in $\delta P$,
\begin{equation}
\label{eq:meluche1}
E(B_1) \gtrsim E(Q_{c_1}) + c_1 \delta P
\end{equation}
and 
\begin{equation}
\label{eq:meluche2}
E(B_2) \gtrsim E(Q_{c_2}) - c_2 \delta P.
\end{equation}
On the other hand, by conservation of the total energy, we have 
$$E(Q_{c_1}) + E(Q_{c_2}) \simeq E(B_1) + E(B_2),$$
so that, by \eqref{eq:meluche2},
\begin{equation}
\label{eq:meluche3}
E(B_1) \lesssim E(Q_{c_1}) + E(Q_{c_2}) - E(B_2) \lesssim E(Q_{c_1}) + c_2 \delta P.
\end{equation}
Combining \eqref{eq:meluche1} and \eqref{eq:meluche3} yields
\begin{equation}
\label{eq:meluche4}
(c_2 - c_1) \delta P \gtrsim 0.
\end{equation}
The last and key observation is a monotonicity formula, which states that in the above configuration necessarily
\begin{equation}
\label{eq:meluche5}
\delta P \lesssim 0.
\end{equation}
Therefore, if the speeds are ordered as in Theorem \ref{thm:orbistab}, namely $c_2 > c_1$, then \eqref{eq:meluche4} may hold only if $\delta P = 0$. Orbital stability of the chain then follows from the orbital stability of the individual solitons.

\begin{remark}
The monotonicity formula \eqref{eq:meluche5} has a simple physical interpretation. Momentum is a signed quantity which, in the present normalization, is non-negative for ``waves'' or ``particles'' travelling to the left and non-positive for those travelling to the right. Since $Q_{c_1, a_1}$ is located at the left of $Q_{c_2, a_2}$ (recall that $a_2 - a_1 \gg 1$ by assumption), the momentum transferred from $Q_{c_1, a_1}$ towards $Q_{c_2, a_2}$ (i.e. $\delta P$) needs to travel from the left to the right,
\footnote{To be more precise, the transferred momentum needs to travel from the left to the right in a reference frame attached to $Q_{c_1, a_1}$, i.e. moving at speed $c_1$. Dispersive waves have a speed larger or equal in absolute value than the speed of sound $\sqrt{2}$ given by the dispersion relation. Since solitons are subsonic, a dispersive (hence sonic or supersonic) wave going to the right in the reference frame attached to a soliton also goes to the right in the frame of the lab. This peculiarity is a major difference between the Gross-Pitaevskii equation and the nonlinear Schr\"odinger equations with zero density at spatial infinity, where the speeds of solitons and dispersive waves may overlap.}
and is therefore non-positive. Equivalently (and alternatively from the point of view of $Q_{c_2, a_2}$), the momentum transferred from $Q_{c_2, a_2}$ towards $Q_{c_1, a_1}$, i.e. $- \delta P$, needs to travel from the right to the left, and is therefore non-negative.
\end{remark}

We are now in position to present the detailed description of the elements in the proof. For later reference, we mention the explicit values
\begin{equation}
\label{def:Pc}
P(Q_c) = \arctan \Big( \frac{c}{\sqrt{2 - c^2}} \Big) + \frac{c \sqrt{2 - c^2}}{2} - {\rm sign}(c)\frac{\pi}{2},
\end{equation}
and
\begin{equation}
\label{def:Ec}
E(Q_c) = \frac{(2 - c^2)^\frac{3}{2}}{3}.
\end{equation} 

\renewcommand{\thesubsubsection}{\Alph{subsubsection}}
\subsubsection{Minimizing properties of solitons}
\label{sub:coer}

We will rely on a quantitative version of the minimizing property \eqref{eq:minimizingP}. Define the quadratic form
$$H_c := E''(Q_c) + c P''(Q_c).$$ 
In explicit form,
\begin{equation}
\label{eq:Hc}
H_c(\eps) := \frac{1}{4} \int_\R \frac{(\partial_x \eps_\eta)^2}{1 - \eta_c} + \frac{1}{4} \int_\R \Big( 2 - \frac{\partial_{xx}^2 \eta_c}{(1 - \eta_c)^2} - \frac{(\partial_x \eta_c)^2}{(1 - \eta_c)^3} \Big) \eps_\eta^2 + \int_\R \Big( (1 - \eta_c) \eps_v^2 + (c - 2 v_c) \eps_\eta \eps_v \Big),
\end{equation}
where we have written $\eps = (\eps_\eta, \eps_v)$. We have

\begin{prop}
\label{prop:linstab}
There exists $\Lambda_c > 0$ such that
\begin{equation}
\label{eq:estimHc}
H_c(\eps) \geq \Lambda_c \|\eps\|_X^2,
\end{equation}
for any $\eps \in X(\R)$ satisfying the orthogonality conditions
\begin{equation}
\label{cond:orthc}
\langle \eps , \partial_x Q_c \rangle_{L^2(\R)} = P'(Q_c)(\eps) = 0.
\end{equation}
Moreover, $\Lambda_c$ is uniformly bounded from below by a positive constant for $c$ in a compact subset of $(- \sqrt{2}, \sqrt{2}) \setminus \{ 0 \}$.
\end{prop}

\begin{remark}
The first orthogonality condition in \eqref{cond:orthc} is related to the invariance with respect to translations of \eqref{HGP}. It may be rephrased according to the equality
$$\langle \eps , \partial_x Q_c \rangle_{L^2} = \frac{1}{2} \frac{d}{dy} \Big( \big\| (\eta, v) - Q_c(\cdot - y) \big\|_{L^2}^2 \Big)_{|y = 0} = 0.$$
It is possible to choose the condition involving the momentum $P$ as an orthogonality condition due to the property
$$\frac{d}{dc} \big( P(Q_c) \big) \neq 0,$$ 
which follows from the explicit value \eqref{def:Pc} of $P(Q_c)$.
\end{remark}

In particular, if $\eps := (\eta, v) - Q_c$ satisfies the orthogonality conditions \eqref{cond:orthc}, then
\begin{equation}
\label{eq:coercive}
(E + c P)(\eta, v) \geq (E + c P)(Q_c) + \Lambda_c \| \eps \|^2_X + \boO \big( \| \eps \|_X^3 \big), 
\end{equation} 
as $\| \eps \|_X \to 0$.

Our next goal is to obtain a coercivity inequality in the spirit of \eqref{eq:coercive} for functions $(\eta, v)$ which do not belong to a small neighbourhood of a single soliton but instead to a neighbourhood of a finite sum of well-separated solitons like the one described by \eqref{cond:init} and \eqref{cond:separ}. We begin with the treatment of the orthogonality conditions in that case. 

\subsubsection{Orthogonal decomposition of a chain of solitons}
\label{sub:modul}

Given a positive parameter $L$, we define the set of admissible positions as
$$\Pos_N(L) := \big\{ \ga := (a_1, \ldots, a_N) \in \R^N, \ {\rm s.t.} \ a_{k + 1} > a_k + L, \ \forall 1 \leq k \leq N - 1 \big\},$$
and, for $\alpha > 0$ and $\gc = (c_1, \ldots, c_N) \in \Adm_N$, we set
$$\boU_\gc(\alpha, L) := \Big\{ (\eta, v) \in X(\R), \ {\rm s.t.} \ \inf_{\ga \in \Pos_N(L)} \big\| (\eta, v) - R_{\gc, \ga} \big\|_X < \alpha \Big\},$$
and
$$\mu_\gc := \min \big\{ |c_k|, \ 1 \leq k \leq N \big\}.$$
When $\alpha$ is small enough and $L$ is sufficiently large, we claim that the set $\boU_\gc(\alpha, L)$ provides a suitable framework to perform satisfactory decompositions. More precisely, we have

\begin{prop}
\label{prop:stadec}
Let $\gc^* \in \Adm_N$. There exist positive numbers $\alpha_1$, $L_1$ and $K_1$, depending only on $\gc^*$, 
and two continuously differentiable functions $\gC \in \boC^1(\boU_{\gc^*}(\alpha_1, L_1), \Adm_N)$ and $\gA \in \boC^1(\boU_{\gc^*}(\alpha_1, L_1), \R^N)$ such that for any $(\eta, v) \in \boU_{\gc^*}(\alpha_1, L_1)$, the function 
$$\eps := (\eta, v) - R_{\gc, \ga},$$
where $\gc = (c_1, \ldots, c_N) := \gC(\eta, v)$ and $\ga = (a_1, \ldots, a_N) := \gA(\eta, v)$, satisfies the orthogonality conditions
\begin{equation}
\label{cond:staorth}
\langle \eps, \partial_x Q_{c_k, a_k} \rangle_{L^2} = P'(Q_{c_k, a_k})(\eps) = 0,
\end{equation}
for any $1 \leq k \leq N$. Moreover, if
$$\big\| (\eta, v) - R_{\gc^*, \ga^*} \big\|_X \leq \alpha,$$
for some $\ga^* = (a_1^*, \ldots, a_N^*) \in \Pos_N(L)$, with $L \geq L_1$ and $\alpha \leq \alpha_1$, then
\begin{equation}
\label{est:stamodul1}
\| \eps \|_X + \sum_{k = 1}^N |c_k - c_k^*| + \sum_{k = 1}^N |a_k - a_k^*| \leq K_1 \alpha.
\end{equation}
\end{prop}

Increasing $L_1$ and decreasing $\alpha_1$ in the statement of Proposition \ref{prop:stadec} if necessary, one may further infer from \eqref{est:stamodul1} the following facts which we will use throughout. 

\begin{cor}
\label{cor:toutvabien}
Under the assumptions of Proposition \ref{prop:stadec}, whenever $(\eta, v) \in \boU_{\gc^*}(\alpha_1, L)$ with $L \geq L_1$, it follows that $(\eta, v) \in \boN\boV(\R)$, since 
\begin{equation}
\label{eq:ouf3}
1 - \eta \geq \frac{\mu_{\gc^*}^2}{8}.
\end{equation}
Moreover, we have
\begin{equation}
\label{est:stamodul2}
\gA(\eta, v) \in \Pos_N \big( L - 1 \big),
\end{equation}
\begin{equation}
\label{eq:ouf1}
\nu_{\gC(\eta, v)} \geq \frac{\nu_{\gc^*}}{2},
\end{equation}
and
\begin{equation}
\label{eq:ouf2}
\mu_{\gC(\eta, v)} \geq \frac{\mu_{\gc^*}}{2}.
\end{equation}
\end{cor}

\subsubsection{Almost minimizing properties close to a sum of solitons}
\label{sub:decomp}

Given $\gc^* := (c_1^*, \ldots, c_N^*) \in \Adm_N$ and $(\eta, v) \in \boU_{\gc^*}(\alpha_1, L_1)$, we write
$$(\eta, v) = R_{\gc, \ga} + \eps,$$
according to Proposition \ref{prop:stadec}, where $\eps$ satisfies the orthogonality conditions \eqref{cond:staorth}. Our next goal is to show that the sum $R_{\gc, \ga}$ possesses almost minimizing properties. Since each soliton $Q_{c_k ,a_k}$ has its own speed $c_k$, the function $E + c P$, whose coercivity properties were exhibited in \eqref{eq:coercive} in the case of a single soliton $Q_c$, has to be replaced by some function which, close to each location point $a_k$, resembles $E + c_k P$.

More precisely, we define the functions
\footnote{In the definition of $\Psi_k$, one could a priori replace the constant $\nu_{\gc^*}/16$ by an arbitrary positive constant $\nu^*$. In order that the monotonicity formula later stated in Proposition \ref{prop:mono} holds, the choice of $\nu^*$ is actually restricted. Up to a constant multiplicity factor, $\nu_{\gc^*}/16$ is presumably optimal.}
$$\Psi_k(x) := \left\{ \begin{array}{lll}
1 & {\rm if} \quad k = 1,\\[5pt]
\displaystyle \frac{1}{2} \Big( 1 + \th \Big( \frac{\nu_{\gc^*}}{16} \Big( x - \frac{a_{k - 1} + a_k}{2} \Big) \Big) & {\rm if} \quad 2 \leq k \leq N,\\[5pt]
0 & {\rm if} \quad k = N + 1.\\
\end{array} \right.$$
We next set
$$F(\eta, v) := E(\eta, v) + \sum_{k = 1}^N c_k P_k(\eta, v),$$
where
$$P_k(\eta, v) := \int_\R p(\eta, v) \big( \Psi_k - \Psi_{k + 1} \big).$$
Notice that by construction,
$$\sum_{k = 1}^N \big( \Psi_k - \Psi_{k + 1} \big) = \Psi_1 - \Psi_{N + 1} = 1,$$
and that, provided the points $a_k$ are sufficiently far apart one from each other, we have for each $k$, $\Psi_k - \Psi_{k + 1} \simeq 1$ in a neighborhood of $Q_{c_k, a_k}$ and $\Psi_k - \Psi_{k + 1} \simeq 0$ in the neighborhood of $Q_{c_j, a_j}$, for $j \neq k$. Notice also that, even though it does not explicitly appear in the notation, $\Psi_k$ depends on $(\eta, v)$ through the points $a_k$. This will be of particular importance when we introduce time dependence later on. 

In order to estimate the function $F$, we localize the function $\eps$ according to the following decomposition. Let $0 < \tau < \nu_{\gc^*}/16$ to be defined later (see \eqref{eq:capasse}), depending only on $\gc^*$. We set
$$\Phi_k(x) := \frac{1}{2} \Big( \th \Big( \tau \Big( x - a_k + \frac{L_1}{4} \Big) \Big) - \th \Big( \tau \Big( x - a_k - \frac{L_1}{4} \Big) \Big) \Big),$$
for $1 \leq k \leq N$, and
$$\Phi_{k, k + 1} := \left\{ \begin{array}{lll}
\frac{1}{2} \Big(1 - \th \big( \tau \big( x - a_1 + \frac{L_1}{4} \big) \big) \Big) & {\rm if} \quad k = 0,\\[5pt]
\frac{1}{2} \Big( \th \big( \tau \big( x - a_k - \frac{L_1}{4} \big) \big) - \th \big( \tau \big( x - a_{k + 1} + \frac{L_1}{4} \big) \big) \Big) & {\rm if} \quad 1 < k < N,\\[5pt]
\frac{1}{2} \Big( 1 + \th \big( \tau \big( x - a_N - \frac{L_1}{4} \big) \big) \Big) & {\rm if} \quad k = N.
\end{array} \right.$$
Once more, we note that by construction,
\begin{equation}
\label{eq:partition}
\sum_{k = 1}^N \Phi_k + \sum_{k = 0}^N \Phi_{k, k + 1} = 1,
\end{equation}
and that each term in the sum is non-negative. We finally define
\begin{equation}
\label{def:epskbis}
\eps_k := \Phi_k(\cdot + a_k)^\frac{1}{2} \eps(\cdot + a_k),
\end{equation}
for $1 \leq k \leq N$, and
\begin{equation}
\label{def:epsk}
\eps_{k, k + 1} := \Phi_{k, k + 1}^\frac{1}{2} \eps,
\end{equation}
for $0 \leq k \leq N$.

\begin{prop}
\label{prop:inutile}
Assume that $(\eta, v) = R_{\gc, \ga} + \eps \in \boU_{\gc^*}(\alpha_1, L)$, with $L \geq L_1$. Then,
\begin{equation}
\label{eq:decE}
\begin{split}
E(\eta, v) = & \sum_{k = 1}^N E(Q_{c_k}) + \frac{1}{2} \sum_{k = 1}^N E''(Q_{c_k})(\eps_k, \eps_k) + \frac{1}{2} \sum_{k = 0}^N E''(0)(\eps_{k, k + 1}, \eps_{k, k + 1})\\
& + \boO \Big( \Big( \tau + \exp \Big( - \tau \frac{L_1}{2} \Big) \Big) \| \eps \|_X^2 \Big) + \boO \Big( \| \eps \|_X^3 \Big) + \boO \Big( L \exp \Big( - \frac{\nu_{\gc^*} L}{2} \Big) \Big),
\end{split}
\end{equation}
and
\begin{equation}
\label{eq:decp}
\begin{split}
P_k(\eta, v) = & P(Q_{c_k}) + \frac{1}{2} P''(Q_{c_k})(\eps_k, \eps_k) + \frac{1}{2}P_k''(0)(\eps_{k, k + 1}, \eps_{k, k + 1}) + \frac{1}{2}P_k''(0)(\eps_{k - 1, k}, \eps_{k - 1, k})\\
& + \boO \Big( L \exp \Big( - \frac{\nu_{\gc^*} L}{16} \Big) \Big) + \boO \Big( \exp \Big(- \frac{\tau L_1}{2} \Big) \| \eps \|_X^2 \Big).
\end{split}
\end{equation}
\end{prop}

\begin{remark}
Here and in the sequel, we have found convenient to simplify the presentation of our estimates by introducing the notation $\boO$. We adopt as a definition that one is allowed to substitute a quantity $S$ by the notation $\boO(T)$ if and only if there exists a positive constant $K$, {\bf depending only on $\gc^*$}, such that
$$|S| \leq K T.$$
\end{remark}

Later on, we will analyze the evolution in time of the evaluation of $F$ along a solution $(\eta, v)$ of \eqref{HGP}. The analysis is complicated by the fact that the speeds $c_k$ are themselves functions of $(\eta, v)$. To circumvent this difficulty, it turns out that it is sufficient to consider instead of $F$ the function $G$ defined on $\boU_{\gc^*}(\alpha_1, L_1)$ by
$$G(\eta, v) := E(\eta, v) + \sum_{k = 1}^N c_k^* P_k(\eta, v).$$
Notice here that the speeds are the reference speeds $c_k^*$. Concerning the functional $G$, we have the following consequence of Proposition \ref{prop:inutile}.

\begin{cor}
\label{cor:pourG}
Assume that $(\eta, v) = R_{\gc, \ga} + \eps \in \boU_{\gc^*}(\alpha_1, L)$, with $L \geq L_1$. Then,
\begin{align*}
G(\eta, v) & = \sum_{k = 1}^N \big( E(Q_{c_k^*}) + c_k^* P(Q_{c_k^*}) \big) + \frac{1}{2} \sum_{k = 1}^N H_{c_k}(\eps_k) + \frac{1}{2} \sum_{k = 0}^N H_{0}^k(\eps_{k, k + 1}) + \boO \Big( \sum_{k = 1}^N |c_k - c_k^*|^2 \Big)\\
& + \boO \Big( \Big( \tau + \exp \Big( - \frac{\tau L_1}{2} \Big) \Big) \| \eps \|_X^2 \Big) + \boO \Big( \| \eps \|_X^3 \Big) + \boO \Big( L \exp \Big( - \frac{\nu_{\gc^*} L}{16} \Big) \Big),
\end{align*}
where we have set
\begin{align*}
& H_0^0(\eps_{0, 1}) = \big( E''(0) + c_1 P_1''(0) \big)(\eps_{0, 1}, \eps_{0, 1}),\\[5pt]
& H_0^k(\eps_{k, k + 1}) = \big( E''(0)+ c_k P_k''(0) + c_{k + 1} P_{k + 1}''(0) \big)(\eps_{k, k + 1},\eps_{k, k + 1}), \quad \forall 1 \leq k \leq N - 1,\\[5pt]
& H_0^N(\eps_{N, N + 1}) = \big( E''(0) + c_N P_N''(0) \big)(\eps_{N, N + 1}, \eps_{N, N + 1}).
\end{align*}
\end{cor}

It is not possible to invoke Proposition \ref{prop:linstab} directly to obtain a lower bound on the terms $H_{c_k}(\eps_k)$ since $\eps_k$ only satisfies the orthogonality conditions \eqref{cond:orthc} asymptotically as $L_1 \to + \infty$. We obtain

\begin{lemma}
\label{lem:coerapp}
Under the assumptions of Corollary \ref{cor:pourG}, there exists a positive constant $\Lambda^*$,
depending only on $\gc^*$, such that
\begin{equation}
\label{eq:miniepsk}
H_{c_k}(\eps_k) \geq \Lambda^* \|\eps_k \|_X^2 + \boO \Big( L_1^\frac{1}{2} \exp \Big( - \frac{\tau L_1}{2} \Big) \| \eps \|_X^2 \Big),
\end{equation}
for any $1 \leq k \leq N$, and such that
\begin{equation}
\label{eq:miniepskbis}
H_0^k(\eps_{k, k + 1}) \geq \Lambda^* \| \eps_{k, k + 1} \|_X^2, 
\end{equation}
for any $0 \leq k \leq N$. Moreover,
\begin{equation}
\label{eq:equivnorm}
\sum_{k = 1}^N \| \eps_k \|_X^2 + \sum_{k = 0}^N \| \eps_{k, k + 1} \|_X^2 \geq \| \eps \|_X^2.
\end{equation}
\end{lemma}

In view of Corollary \ref{cor:pourG} and Lemma \ref{lem:coerapp}, we now fix the value of $\tau$, sufficiently small, and then increase the value of $L_1$, if necessary, in such a way that
\footnote{By the left-hand side of \eqref{eq:capasse}, we actually mean all the terms which we have previously estimated in Corollary \ref{cor:pourG} and Lemma \ref{lem:coerapp} as the corresponding $\boO$.} 
\begin{equation}
\label{eq:capasse}
\boO \Big( \Big( \tau + L_1^\frac{1}{2} \exp \Big( - \frac{\tau L_1}{2} \Big) \Big) \| \eps \|_X^2 \Big) \leq \frac{\Lambda^*}{4} \| \eps \|_X^2.
\end{equation}
With these choices we obtain 

\begin{cor}
\label{cor:estG}
Assume that $(\eta, v) = R_{\gc, \ga} + \eps \in \boU_{\gc^*}(\alpha_1, L)$, with $L \geq L_1$. Then,
\begin{equation}
\label{eq:lowerG}
\begin{split}
G(\eta, v) & \geq \sum_{k = 1}^N \big( E(Q_{c_k^*}) + c_k^* P(Q_{c_k^*}) \big) + \frac{\Lambda^*}{4} \| \eps \|_X^2 + \boO \Big( \| \eps \|_X^3 \Big) + \boO \Big( \sum_{k = 1}^N |c_k - c_k^*|^2 \Big)\\
& + \boO \Big( L \exp \Big( - \frac{\nu_{\gc^*} L}{16} \Big) \Big),
\end{split}
\end{equation} 
as well as
\begin{equation}
\label{eq:upperG}
G(\eta, v) \leq \sum_{k = 1}^N \big( E(Q_{c_k^*}) + c_k^* P(Q_{c_k^*}) \big) + \boO \Big( \| \eps \|_X^2 \Big) + \boO \Big( \sum_{k = 1}^N |c_k - c_k^*|^2 \Big) + \boO \Big( L \exp \Big( - \frac{\nu_{\gc^*} L}{16} \Big) \Big).
\end{equation}
\end{cor}

Notice that, up to this point, the entire analysis has been independent of time. In the next section, we combine the previous results with information extracted from equation \eqref{HGP}.

\subsubsection{Dynamics of the modulation parameters}
\label{sub:evol}

In this subsection, as well as in Subsections \ref{sub:mono} and \ref{sub:conteps} below, $\gc^* \in \Adm_N$ is fixed and we assume that $T > 0$, and $(\eta, v) \in \boC^0([0, T], X(\R))$ is a given solution to equation \eqref{HGP} such that for some $\alpha \leq \alpha_1$ and $L \geq L_1$,
$$\big( \eta(\cdot, t), v(\cdot, t) \big) \in \boU_{\gc^*}(\alpha, L),$$
for any $t \in [0, T]$.

For $t \in [0, T]$, we define
$$\gc(t) = (c_1(t), \ldots, c_N(t)) := \gC \big( v(\cdot, t), \eta(\cdot, t) \big) \quad {\rm and} \quad \ga(t) = (a_1(t), \ldots, a_N(t)) := \gA \big( v(\cdot, t), \eta(\cdot, t) \big),$$
and
$$\eps(\cdot,t) := \big( \eta(\cdot, t), v(\cdot, t) \big) - R_{\gc(t), \ga(t)},$$
where the functions $\gC$ and $\gA$ are given by Proposition \ref{prop:stadec}. 

The following proposition expresses the fact that the modulation parameters should follow those of the underlying solitons, as long as the solution remains close to them. 

\begin{prop}
\label{prop:modul}
There exist positive numbers $\alpha_2 \leq \alpha_1$ and $L_2 \geq L_1$, depending only on $\gc^*$, such that, if $\alpha \leq \alpha_2$ and $L \geq L_2$, then $\gc \in \boC^1([0, T], \Adm_N)$, $\ga \in \boC^1([0, T], \R^N)$, and we have the estimates 
\begin{equation}
\label{est:modul2}
\sum_{k = 1}^N \Big( |a_k'(t) - c_k(t)| + |c_k'(t)| \Big) = \boO \Big( \| \eps(\cdot, t)\|_X \Big) + \boO \Big( L \exp \Big( - \frac{\nu_{\gc^*} L}{2} \Big) \Big),
\end{equation}
for any $t \in [0, T]$.
\end{prop}

From now on, we make the additional assumption that our base set of speeds $\gc^*$ satisfies the inequalities 
\begin{equation}
\label{eq:ordo}
\bc_ {\bf 1}^{\bf*} {\bf < \cdots <} \bc_{\bN}^{\bf *}{\bf ,}
\end{equation}
and we set
$$\sigma^* = \frac{1}{2} \min \big\{ c_{k + 1}^* - c_k^*, \ 1 \leq k \leq N-1 \big\} > 0.$$
As a consequence of \eqref{est:stamodul1}, \eqref{est:modul2} and \eqref{eq:ordo}, we obtain 

\begin{cor}
\label{cor:separated}
There exist positive numbers $\alpha_3 \leq \alpha_2$ and $L_3 \geq L_2$, depending only on $\gc^*$, such that if $\alpha \leq \alpha_3$ and $L \geq L_3$, then for any $t \in [0, T]$,
\begin{equation}
\label{est:stamodul2bis}
a_{k + 1}(t) - a_k(t) > a_{k+1}(0)- a_k(0) + \sigma^*t \geq (L - 1) + \sigma^* t,
\end{equation}
for each $1 \leq k \leq N - 1$, and 
\begin{equation}
\label{def:Wtierce}
\sqrt{2 - (a_k'(t))^2} \geq \frac{\nu_{\gc^*}}{2},
\end{equation}
for each $1 \leq k \leq N$.
\end{cor}

\subsubsection{Evolution in time of $\boG$ and monotonicity formula}
\label{sub:mono}

We define the function $\boG$ on $[0, T]$ by
$$\boG(t) := G \big( \eta(\cdot, t), v(\cdot, t) \big),$$
where $G$ was introduced in Subsection \ref{sub:decomp}. We rewrite $\boG$ according to the equality
\begin{equation}
\label{eq:decompG}
\boG(t) = E(t) + c_1^* P(t) + \sum_{k = 2}^N \big( c_k^* - c_{k - 1}^* \big) Q_k(t), 
\end{equation}
where
$$E(t) := E(\eta(\cdot, t), v(\cdot, t)), \quad P(t) := P(\eta(\cdot, t), v(\cdot, t)),$$
and
$$Q_k(t) := \int_\R \Psi_k(x) p \big( \eta(x, t), v(x, t) \big) dx = \sum_{j = k}^N P_j(t).$$
In case $N = 1$ (i.e. there is only one soliton), it follows from the conservation of energy and momentum that $\boG$ is constant in time. In the general case, we have the following monotonicity formula.
 
\begin{prop}
\label{prop:mono}
There exist positive numbers $\alpha_4 \leq \alpha_3$ and $L_4 \geq L_3$, depending only on $\gc^*$, such that if $\alpha \leq \alpha_4$ and $L \geq L_4$, then 
$$\frac{d}{dt} \Big( Q_k(t) \Big) \leq \boO \Big( \exp \Big( - \frac{\nu_{\gc^*} (L + \sigma^* t)}{16} \Big) \Big),$$
for any $t \in [0, T]$. 
\end{prop}

As a consequence of the ordering condition \eqref{eq:ordo}, we thus obtain

\begin{cor}
\label{cor:mono}
If $\alpha \leq \alpha_4$ and $L \geq L_4$, then
$$\frac{d}{dt} \Big( \boG(t) \Big) \leq \boO \Big( \exp \Big( - \frac{\nu_{\gc^*} (L + \sigma^* t)}{16} \Big) \Big),$$
for any $t \in [0, T]$.
\end{cor}

\subsubsection{Uniform control on $\| \eps \|_X$}
\label{sub:conteps}

We are now in position to obtain a uniform control on $\| \eps(\cdot, t) \|_X$ combining the previous results with simple algebraic identities. We divide the analysis in a number of steps, each one providing a uniform control on an intermediate quantity.

\begin{step}
Control on $|Q_k(t) - Q_k(0)|$.
\end{step}

Since $E$ and $P$ are conserved by the flow, we obtain from \eqref{eq:decompG},
\begin{equation}
\label{eq:tr1}
Q_j(0) - Q_j(t) = \boG(0) - \boG(t) + \frac{1}{c_j^* - c_{j - 1}^*} \sum_{k \neq j} (c_k^* - c_{k - 1}^*) \big( Q_k(t) - Q_k(0) \big),
\end{equation}
for any $2 \leq j \leq N$. On the one hand, we have by \eqref{eq:upperG} for $\boG(0)$, and \eqref{eq:lowerG} for $\boG(t)$,
\begin{equation}
\label{eq:tr2}
\boG(0) - \boG(t) \leq \boO \Big( \| \eps(\cdot, 0) \|_X^2 \Big) + \boO \Big( \| \eps(\cdot, t) \|_X^3 \Big) + \boO \Big( \sum_{k = 1}^N |c_k(t) - c_k^*|^2 \Big) + \boO \Big( L \exp \Big( - \frac{\nu_{\gc^*} L}{16} \Big) \Big). 
\end{equation}
On the other hand, from Proposition \ref{prop:mono} and assumption \eqref{eq:ordo} on the speeds, we have 
\begin{equation}
\label{eq:tr3}
\frac{1}{c_j^* - c_{j - 1}^*} \sum_{k \neq j} (c_k^* - c_{k - 1}^*) \big( Q_k(t) - Q_k(0) \big) \leq \boO \Big( \exp \Big( - \frac{\nu_{\gc^*} L}{16} \Big) \Big).
\end{equation}
Combining \eqref{eq:tr1}, \eqref{eq:tr2} and \eqref{eq:tr3} for the upper bound part, and using once more Proposition \ref{prop:mono} for the lower bound part, we derive
\begin{equation}
\label{eq:Pj}
\sum_{k = 2}^N \big| Q_k(t) - Q_k(0) \big| = \boO \Big( \| \eps(\cdot, 0) \|_X^2 \Big) + \boO \Big( \|\eps(\cdot, t) \|_X^3 \Big) + \boO \Big( \sum_{k = 1}^N |c_k(t) - c_k^*|^2 \Big) + \boO \Big( L \exp \Big( - \frac{\nu_{\gc^*} L}{16} \Big) \Big).
\end{equation}

\begin{step}
Control on $|P(Q_{c_k(t)}) - P(Q_{c_k(0)})|$.
\end{step}

It follows from Proposition \ref{prop:inutile} that 
\begin{equation}
\label{eq:fart0}
P(t) - Q_2(t) = P_1(t) = P(Q_{c_1(t)}) + \boO \Big( \| \eps(\cdot, t) \|_X^2 \Big) + \boO \Big( L \exp(- \frac{\nu_{\gc^*} L}{16} \Big) \Big), 
\end{equation}
and that for $2 \leq k \leq N$,
\begin{equation}
\label{eq:fart1}
Q_k(t) - Q_{k + 1}(t) = P_k(t) = P(Q_{c_k(t)}) + \boO \Big( \| \eps(\cdot, t) \|_X^2 \Big) + \boO \Big( L \exp \Big( - \frac{\nu_{\gc^*} L}{16} \Big) \Big).
\end{equation}
Combining \eqref{eq:fart0} and \eqref{eq:fart1} with \eqref{eq:Pj} and the conservation of the momentum $P$, we obtain
\begin{equation}
\label{eq:fart2}
\begin{split}
\sum_{k = 1}^N \big| P(Q_{c_k(t)}) - P(Q_{c_k(0)}) \big| & = \boO \Big( \| \eps(\cdot, 0) \|_X^2 \Big) + \boO \Big( \| \eps(\cdot, t) \|_X^2 \Big) + \boO \Big( \sum_{k = 1}^N |c_k(t) - c_k^*|^2 \Big)\\
& + \boO \Big( L \exp \Big( - \frac{\nu_{\gc^*} L}{16} \Big) \Big).
\end{split}
\end{equation}

\begin{step}
Control on $|c_k(t) - c_k^*|$.
\end{step}

Combining \eqref{eq:fart2} with the fact that
$$\frac{d}{dc} \Big( P(Q_c) \Big) = (2 - c^2)^\frac{1}{2} \neq 0,$$
and the inequality
$$\sum_{k = 1}^N \big| c_k(t) - c_k^* \big| \leq \sum_{k = 1}^N \big| c_k(t) - c_k(0) \big| + \sum_{k = 1}^N \big| c_k(0) - c_k^* \big|,$$
we obtain the existence of positive numbers $\alpha_5 \leq \alpha_4$ and $L_5 \geq L_4$, depending only on $\gc^*$, such that if $\alpha \leq \alpha_5$ and $L \geq L_5$, then
$$\sum_{k = 1}^N \big| c_k(t) - c_k^* \big| = \boO \Big( \| \eps(\cdot, 0) \|^2 \Big) + \boO \Big( \| \eps(\cdot, t) \|^2 \Big) + \boO \Big( \sum_{k = 1}^N |c_k(0) - c_k^*| \Big) + \boO \Big( L \exp \Big( - \frac{\nu_{\gc^*} L}{16} \Big) \Big).$$
Finally, from Proposition \ref{prop:stadec} we may bound $\underset{k = 1}{\overset{N}{\sum}} |c_k(0) - c_k^*|$ by $\boO(\| \eps(\cdot, 0) \|_X)$, and therefore obtain
\begin{equation}
\label{eq:contrckbis}
\sum_{k = 1}^N \big| c_k(t) - c_k^* \big| = \boO \Big( \| \eps(\cdot, 0) \| \Big) + \boO \Big( \|\eps(\cdot, t) \|^2 \Big) + \boO \Big( L \exp \Big(- \frac{\nu_{\gc^*} L}{16} \Big) \Big).
\end{equation}

\begin{step}
Control on $\| \eps(\cdot, t) \|_X$.
\end{step}

Combining Corollary \ref{cor:estG} for $(\eta, v) := (\eta(\cdot, t), v(\cdot, t))$, and Corollary \ref{cor:mono}, we obtain, if $\alpha \leq \alpha_4$ and $L \geq L_4$,
$$\frac{\Lambda^*}{4} \| \eps(\cdot, t) \|_X^2 = \boO \Big( \| \eps(\cdot, 0) \|_X^2 \Big) + \boO \Big( \sum_{k = 1}^N |c_k(t) - c_k^*|^2 \Big) + \boO \Big( L \exp \Big( - \frac{\nu_{\gc^*} L}{16} \Big) \Big) + \boO \Big( \| \eps(\cdot, t) \|_X^3 \Big).$$ 
It follows therefore from \eqref{eq:contrckbis} that there exist positive numbers $\alpha_6 \leq \alpha_5$ and $L_6 \geq L_5$, depending only on $\gc^*$, such that if $\alpha \leq \alpha_6$ and $L \geq L_6$, then
\begin{equation}
\label{eq:fond0}
\| \eps(\cdot, t) \|_X^2 \leq \boO \Big( \| \eps(\cdot, 0) \|_X^2 \Big) + \boO \Big( L \exp \Big( - \frac{\nu_{\gc^*} L}{16} \Big) \Big). 
\end{equation}

\subsubsection{Proof of Theorem \ref{thm:orbistab} completed}
\label{sub:end}

Let $(\eta^0, v^0)$ be as in the statement of Theorem \ref{thm:orbistab}. We first impose the condition $\alpha^* < \alpha_6$ and $L^* > L_6 + 2$, so that by continuity of the flow, since $L^0 \geq L^*$ by assumption, 
$$T_{\rm stop} := \sup \big\{ t \geq 0, \ {\rm s.t.} \ (\eta(\cdot, s), v(\cdot, s)) \in \boU_{\gc^*}(\alpha_6, L^0 - 2), \ \forall s \in [0, t] \big\} > 0.$$
Notice that $T_{\rm stop} < T_{\rm max}$ since $\boU_{\gc^*}(\alpha_6, L^0 - 2) \subset \boN\boV(\R)$.

By \eqref{eq:fond0} we have
\begin{equation}
\label{eq:yes0}
\big\| \eps(\cdot, t) \big\|_X = \boO \Big( \alpha^0 + \exp \Big( - \frac{\nu_{\gc^*} L^0}{33} \Big) \Big).
\end{equation}
Combining \eqref{eq:yes0} with \eqref{eq:contrckbis} and \eqref{eq:fond0} yields a positive number $K_6$, depending only on $\gc^*$, such that 
\begin{equation}
\label{eq:yes}
\big\| (\eta(\cdot, t), v(\cdot, t)) - R_{\gc^*, \ga(t)} \big\|_X \leq K_6 \Big( \alpha^0 + \exp \Big( - \frac{\nu_{\gc^*} L^0}{33} \Big) \Big), 
\end{equation}
for any $t \in [0, T_{\rm stop})$. On the other hand, by definition of $L^0$, we have $(\eta(0), v(0)) \in \boU_{\gc^*}(\alpha^*, L^0)$. In particular, by Proposition \ref{prop:stadec},
$$\min \big\{ a_{k + 1}(0) - a_k(0), \ 1 \leq k \leq N - 1 \big\} \geq L^0 - 1,$$
so that, by Corollary \ref{cor:separated}, we finally have
\begin{equation}
\label{eq:yesbis}
\min \big\{ a_{k + 1}(t) - a_k(t), \ 1 \leq k \leq N - 1 \big\} = \min \big\{ a_{k + 1}(0) - a_k(0), \ 1 \leq k \leq N - 1 \big\} > L^0 - 2,
\end{equation}
for any $t \in [0, T_{\rm stop})$.

We therefore additionally impose the condition that $\alpha^*$ and $L^*$ satisfy 
$$K_6 \Big( \alpha^* + \exp \Big( - \frac{\nu_{\gc^*} L^*}{33} \Big) \Big) < \alpha_6,$$
so that
$$T_{\rm stop} = + \infty,$$
and in particular, \eqref{eq:yes} and \eqref{eq:yesbis} hold for any $t \in [0, + \infty)$. Finally, combining Proposition \ref{prop:modul} and \eqref{eq:contrckbis}, we obtain
$$\sum_{k = 1}^N |a_k'(t) - c_k^*| \leq K_7 \Big( \alpha^0 + \exp \Big( - \frac{\nu_{\gc^*} L^0}{33} \Big) \Big),$$
and the proof of Theorem \ref{thm:orbistab} is completed by setting $A^* = \max \{ K_6, K_7 \}$.

\subsection{Outline of the paper} 

In Section \ref{sec:minprop}, we present the proofs of the results stated in Subsections \ref{sub:coer}, \ref{sub:modul} and \ref{sub:decomp} of the introduction, which are related to the minimizing properties of solitons and sums of solitons. Section \ref{sec:evol} is instead devoted to the proofs of the results related to the dynamical properties of \eqref{HGP}, and which are stated in Subsections \ref{sub:evol} and \ref{sub:mono} of the introduction. Finally, for the sake of completeness, we provide in Appendix \ref{sec:implicit} a quantitative version of the implicit function theorem which we use in Section \ref{sec:minprop} in order to define the modulation parameters.

\numberwithin{cor}{section}
\numberwithin{equation}{section}
\numberwithin{lemma}{section}
\numberwithin{prop}{section}
\numberwithin{remark}{section}
\numberwithin{theorem}{section}
\section{Minimizing properties around a soliton and a sum of solitons}
\label{sec:minprop}

\subsection{Minimizing properties of solitons}

In this subsection we present the proof of Proposition \ref{prop:linstab}. It is reminiscent from arguments developed in \cite{LinZhiw1}.

\begin{proof}[Proof of Proposition \ref{prop:linstab}] 
Let us recall that the quadratic form $H_c$ is defined in \eqref{eq:Hc} as
$$H_c(\eps) := \frac{1}{4} \int_\R \frac{(\partial_x \eps_\eta)^2}{1 - \eta_c} + \frac{1}{4} \int_\R \Big( 2 - \frac{\partial_{xx}^2 \eta_c}{(1 - \eta_c)^2} - \frac{(\partial_x \eta_c)^2}{(1 - \eta_c)^3} \Big) \eps_\eta^2 + \int_\R \Big( (1 - \eta_c) \eps_v^2 + (c - 2 v_c) \eps_\eta \eps_v \Big).$$
Equation \eqref{eq:crit} writes for $\eta_c$ as
\begin{equation}
\label{eq:etac}
\partial_{xx}^2 \eta_c - (2 - c^2) \eta_c + 3 \eta_c^2 = 0. 
\end{equation}
Invoking \eqref{form:etavc} and \eqref{eq:etac}, the expression of $H_c$ may be recast as
\begin{equation}
\label{dussautoir}
H_c(\eps) = L_c(\eps_\eta) + \int_\R (1 - \eta_c) \Big( \eps_v + \frac{c}{2(1 - \eta_c)^2} \eps_\eta \Big)^2,
\end{equation}
where
$$L_c(\eps_\eta) := \frac{1}{4} \int_\R \frac{(\partial_x \eps_\eta)^2}{1 - \eta_c} + \frac{1}{4} \int_\R \frac{2 - c^2 - 6 \eta_c + 3 \eta_c^2}{(1 - \eta_c)^2} \eps_\eta^2.$$
The Sturm-Liouville operator $\boL_c$ associated to the quadratic form $L_c$, namely
$$\boL_c(\eps_\eta) : = - \partial_x \Big( \frac{\partial_x \eps_\eta}{4 (1 - \eta_c)} \Big) + \frac{2 - c^2 - 6 \eta_c + 3 \eta_c^2}{4 (1 - \eta_c)^2} \eps_\eta,$$
is a self-adjoint, unbounded operator on $L^2(\R)$, with domain $H^2(\R)$. Since $\eta_c(x) \to 0$ as $|x| \to + \infty$ by \eqref{form:etavc}, it follows from the Weyl criterion that its essential spectrum is equal to $[(2 - c^2)/4, + \infty)$. On the other hand, one can translate the invariance with respect to translations of equation \eqref{HGP} into the property that the function $\partial_x \eta_c$ belongs to the kernel of $\boL_c$. Again in view of \eqref{form:etavc}, the function $\partial_x \eta_c$ has exactly one zero. As a consequence, one can infer from standard Sturm-Liouville theory (see e.g. \cite{DunfSch0}) that the operator $\boL_c$ has exactly one negative eigenvalue $- \mu^-$. Moreover, the associated eigenspace, as well as the kernel of $\boL_c$, are of dimension one. We will denote $\eta^-$ an eigenfunction of $\boL_c$ for the eigenvalue $- \mu^-$. Notice in particular, that there exists a constant $0 < \mu^+ \leq 2 - c^2$ such that
\begin{equation}
\label{hari}
L_c(\eps_\eta) \geq \mu^+ \| \eps_\eta \|_{L^2}^2,
\end{equation}
as soon as $\langle \eps_\eta, \eta^- \rangle_{L^2} = \langle \eps_\eta, \partial_x \eta_c \rangle_{L^2} = 0$. 

Coming back to the quadratic form $H_c$, we notice that the associated operator $\boH_c$, namely
$$\boH_c(\eps) := \Big( \boL_c(\eps_\eta) + \frac{c^2}{4(1 - \eta_c)^3} \eps_\eta + \frac{c}{2(1 - \eta_c)} \eps_v, \frac{c}{2(1 - \eta_c)} \eps_\eta + (1 - \eta_c) \eps_v \Big),$$
is a self-adjoint, unbounded operator on $L^2(\R)^2$, with domain $H^2(\R) \times L^2(\R)$. Moreover, it again follows from the Weyl criterion that its essential spectrum is equal to $[(2 - c^2)/(3 + \sqrt{1 + 4 c^2}), + \infty)$. In view of \eqref{form:etavc} and \eqref{dussautoir}, we next check that
$$H_c \Big( \eta^-, - \frac{c \eta^-}{2(1 - \eta_c)^2} \Big) < 0, \quad {\rm and} \quad H_c \big( \partial_x \eta_c, \partial_x v_c \big) = 0,$$
so that $\boH_c$ has at least one negative eigenvalue, and the dimension of its kernel is at least one. Assume next that either $\boH_c$ owns another negative eigenvalue, or its kernel is at least of dimension two. Then, there exists a non-positive direction $\eps = (\eps_\eta, \eps_v)$ for $H_c$ such that $\langle \eps_\eta, \eta^- \rangle_{L^2(\R)} = \langle \eps_\eta, \partial_x \eta_c \rangle_{L^2(\R)} = 0$. In view of \eqref{dussautoir}, this is in contradiction with \eqref{hari}. Therefore, $\boH_c$ has exactly one negative eigenvalue $- \lambda^-$, with eigenfunction $\chi^- := (\chi_\eta^-, \chi_v^-)$, while its kernel is spanned by $\partial_x Q_c = (\partial_x \eta_c, \partial_x v_c)$. As a consequence, there exists a positive constant $\lambda^+ \leq (2 - c^2)/(3 + \sqrt{1 + 4 c^2})$ such that
\begin{equation}
\label{oued}
H_c(\eps) \geq \lambda^+ \| \eps \|_{L^2 \times L^2}^2 := \lambda^+ \big( \| \eps_\eta \|_{L^2}^2 + \| \eps_v \|_{L^2}^2 \big),
\end{equation}
for any pair $\eps$ in the closed subspace
\begin{equation}
\label{lakafia}
P_c^+ := \Big\{ (\eps_\eta, \eps_v) \in H^1(\R) \times L^2(\R), \ {\rm s.t.} \ \langle \eps, \chi^- \rangle_{L^2 \times L^2} = \langle \eps, \partial_x Q_c \rangle_{L^2 \times L^2} = 0 \Big\}.
\end{equation}

The next step in the proof consists in checking that estimate \eqref{oued} remains available (up to a further choice of $\lambda^+$), when the orthogonality conditions in \eqref{lakafia} are replaced by conditions \eqref{cond:orthc}. The proof is reminiscent from \cite{GriShSt1}. We consider the map $S(c) := E(Q_c) + c P(Q_c)$, which is well-defined and smooth on $(- \sqrt{2}, \sqrt{2}) \setminus \{ 0 \}$ in view of \eqref{def:Pc} and \eqref{def:Ec}. Using \eqref{eq:crit}, we compute
$$S'(c) = E'(Q_c)(\partial_c Q_c) + c P'(Q_c)(\partial_c Q_c) + P(Q_c) = P(Q_c),$$
so that, by \eqref{def:Pc},
$$S''(c) = P'(Q_c)(\partial_c Q_c) = (2 - c^2)^\frac{1}{2}.$$
On the other hand, taking the derivative with respect to $c$ of the identity $E'(Q_c) + c P'(Q_c) = 0$, we get
\begin{equation}
\label{picamoles}
\langle \boH_c(\partial_c Q_c), \eps \rangle_{L^2 \times L^2} + P'(Q_c)(\eps) = 0,
\end{equation}
for any $\eps \in X(\R)$. For $\eps = \partial_c Q_c$, this gives
\begin{equation}
\label{warburton}
H_c(\partial_c Q_c) = - P'(Q_c)(\partial_c Q_c) = - (2 - c^2)^\frac{1}{2} < 0.
\end{equation}
At this stage, we can decompose $\partial_c Q_c$ as $\partial_c Q_c = \alpha \chi^- + \beta \partial_x Q_c + r_c$, with $r_c \in P_c^+$. As a consequence,
$$H_c(\partial_c Q_c) = - \lambda_- \alpha^2 + H_c(r_c),$$
so that inequality \eqref{warburton} is equivalent to the fact that $\alpha \neq 0$, combined with the existence of a number $0 \leq \delta < 1$ such that
\begin{equation}
\label{bonnaire}
H_c(r_c) = \delta \lambda^- \alpha^2.
\end{equation}

According to the value of $\delta$, two situations can occur. If $\delta = 0$, then $r_c$ is equal to $0$. Since $\alpha \neq 0$, conditions \eqref{cond:orthc} are actually identical to the conditions in \eqref{lakafia}, so that \eqref{oued} remains available under conditions \eqref{cond:orthc}. If $\delta > 0$, we can decompose a pair $\eps$ which satisfies conditions \eqref{cond:orthc}, as $\eps := a \chi^- + r$, with $r \in P_c^+$, and compute
$$H_c(\eps) = - \lambda^- a^2 + \frac{1 -\delta}{2} H_c(r) + \frac{1 + \delta}{2} H_c(r).$$
In view of \eqref{oued}, the quadratic form $H_c$ is positive on $P_c^+$. Therefore, we can apply the Cauchy-Schwarz inequality to write
$$H_c(\eps) \geq - \lambda^- a^2 + \frac{(1 + \delta) \langle \boH_c(r_c), r \rangle_{L^2 \times L^2}^2}{2 H_c(r_c)} + \frac{1 - \delta}{2} H_c(r).$$
At this stage, it follows from \eqref{cond:orthc} and \eqref{picamoles} that
$$\langle \boH_c(r_c), r \rangle_{L^2 \times L^2} = \langle \boH_c(\partial_c Q_c), \eps \rangle_{L^2 \times L^2} + \lambda^- a \alpha = \lambda^- a \alpha.$$
Combining with \eqref{bonnaire}, we are led to
$$H_c(\eps) \geq \frac{1 - \delta}{2} \Big( \lambda^- a^2 + H_c(r) \Big),$$
so that by \eqref{oued},
$$H_c(\eps) \geq \frac{1 - \delta}{2} \min \big\{ \lambda^-, \lambda^+ \big\} \big\| \eps \big\|_{L^2 \times L^2}^2.$$
As a conclusion, up to a possible further choice of the positive constant $\lambda^+$, estimate \eqref{oued} still holds when $\eps$ satisfies conditions \eqref{cond:orthc}.

In order to complete the proof of estimate \eqref{eq:estimHc}, it remains to replace the $L^2$-norm of $\eps_\eta$ in \eqref{oued} by an $H^1$-norm. Given a number $0 < \tau < 1$, we write
$$H_c(\eps) = \frac{\tau}{4} \int_\R \frac{(\partial_x \eps_\eta)^2}{1 - \eta_c} + \tau \Big( H_c(\eps) - \frac{1}{4} \int_\R \frac{(\partial_x \eps_\eta)^2}{1 - \eta_c} \Big) + (1 - \tau) H_c(\eps).$$
In view of formulae \eqref{form:etavc} and \eqref{eq:Hc}, there exists a positive constant $K_c$, possibly depending on $c$, such that
$$\bigg| H_c(\eps) - \frac{1}{4} \int_\R \frac{(\partial_x \eps_\eta)^2}{1 - \eta_c} \bigg| \leq K_c \big\| \eps \big\|_{L^2 \times L^2}^2,$$
so that, under conditions \eqref{cond:orthc}, we can take benefit of \eqref{oued} to get
$$H_c(\eps) \geq \frac{\tau}{4} \int_\R \frac{(\partial_x \eps_\eta)^2}{1 - \eta_c} + \big( (1 - \tau) \lambda^+ - \tau K \big) \big\| \eps \big\|_{L^2 \times L^2}^2.$$
Since $\eta_c(0) = (2 - c^2)/2 = \max \{ |\eta_c(x)|, \ x \in \R \}$ by \eqref{form:etavc}, we deduce that
$$H_c(\eps) \geq \frac{\tau}{2 c^2} \int_\R (\partial_x \eps_\eta)^2 + \big( (1 - \tau) \lambda^+ - \tau K \big) \big\| \eps \big\|_{L^2 \times L^2}^2.$$
Choosing $\tau$ sufficiently small yields estimate \eqref{eq:estimHc}. Finally, since $H_c$ depends analytically on $c$ and the first two eigenvalues of $H_c$ are simple, the optimal constant in \eqref{eq:estimHc} depends continuously on $c$, from which the last statement in Proposition \ref{prop:linstab} follows. 
\end{proof}

\subsection{Orthogonal decomposition of a chain of solitons}

We now turn to the proof of Proposition \ref{prop:stadec}. One key property in the proof is the exponential decay of the functions $\eta_c$ and $v_c$. As a matter of fact, one can derive from formulae \eqref{form:etavc} that, given any integer $p$, there exists a positive constant $K$, depending only on $p$, such that
\begin{equation}
\label{est:etavc}
\sum_{0 \leq j + k \leq p} \Big( |\partial_x^j \partial_c^k \eta_c(x)| + c^{1 + 2 j + 2 k} |\partial_x^j \partial_c^k v_c(x)| \Big) \leq K (2 - c^2) \exp \big( - (2 - c^2)^\frac{1}{2} |x| \big).
\end{equation}
for any $0 < |c| < \sqrt{2}$ and $x \in \R$. A crucial consequence of \eqref{est:etavc} is the property that two solitons with same speed, as well as two sums of solitons with same speeds, can be closed in $X(\R)$ only if their center(s) of mass are closed. More precisely, we have

\begin{lemma}
\label{lem:diffR}
Let $\gc = (c_1, \ldots, c_N) \in \Adm_N$. Given a positive number $\delta$, there exist two positive numbers $\beta$ and $M$ such that, if
$$\| R_{\gc, \ga} - R_{\gc, \gb} \|_X < \beta,$$
for some positions $\ga \in \Pos_N(M)$ and $\gb \in \Pos_N(M)$, then,
$$\sum_{k = 1}^N |a_k - b_k| < \delta.$$
\end{lemma}

\begin{proof}
The proof is by induction on the integer $N$. When $N = 1$, we have

\setcounter{step}{0}
\begin{step}
\label{D1}
Given a positive number $\delta$, there exists a positive number $\alpha$ such that, if
$$\| Q_{c, a} - Q_{c, b} \|_X < \alpha,$$
then,
$$|a - b| < \delta.$$
\end{step}

The proof is by contradiction. Assuming that Step \ref{D1} is false, there exist a positive number $\delta$ and a sequence $(a_n)_{n \in \N} \in \R^\N$ such that
\begin{equation}
\label{nallet}
\| Q_c(\cdot - a_n) - Q_c \|_X \to 0,
\end{equation}
as $n \to + \infty$, and simultaneously, $|a_n| \geq \delta$ for any $n \in \N$. If the sequence $(a_n)_{n \in \N}$ were unbounded, then, up to some subsequence, it would tend to $+ \infty$, or $- \infty$. In any case, taking the limit $n \to + \infty$ into \eqref{nallet}, we would get that
$2 \| Q_c \|_X = 0,$
which is not possible. As a consequence, the sequence $(a_n)_{n \in \N}$ is bounded. Up to a subsequence, it converges to a real number $a$ such that $|a| \geq \delta$, and simultaneously,
$$\| Q_c(\cdot - a) - Q_c \|_X = 0.$$
This identity implies that $Q_c(\cdot - n a) = Q_c$ for any $n \in \N$. Since $a \neq 0$, this ensures, taking the limit $n \to + \infty$ and invoking \eqref{est:etavc}, that $Q_c = 0$, which provides the desired contradiction.

\begin{step}
End of the proof.
\end{step}

We now assume that the conclusions of Lemma \ref{lem:diffR} are available up to the integer $N - 1$, and that they are not for the integer $N$. Then, there exist a positive number $\delta$, a sequence $(L_j)_{j \in \N}$ tending to $+ \infty$ as $j \to + \infty$, and two sequences $(\ga^{(n, j)})_{(n, j) \in \N^2}$ and $(\gb^{(n, j)})_{(n, j) \in \N^2}$ such that
$$\ga^{(n, j)} \in \Pos_N(L_j) \quad {\rm and} \quad \gb^{(n, j)} \in \Pos_N(L_j),$$
for any $n \in \N$,
\begin{equation}
\label{thion1}
\| R_{\gc, \ga^{(n, j)}} - R_{\gc, \gb^{(n, j)}} \|_{X(\R)} \to 0,
\end{equation}
as $n \to + \infty$, and simultaneously,
\begin{equation}
\label{thion2}
\sum_{k = 1}^N \big| a_k^{(n, j)} - b_k^{(n, j)} \big| > \delta.
\end{equation}
Without loss of generality, we can moreover assume that $a_N^{(n, j)} \leq b_N^{(n, j)} = 0$.

In this situation, if the sequence $(a_N^{(n, j)})_{n \in \N}$ were unbounded for an integer $j$, then, up to a subsequence, it would tend to $- \infty$ as $n \to + \infty$. At this stage, we can invoke \eqref{est:etavc} to certify the existence of a universal constant $K$ such that
\begin{align*}
\| R_{\gc, \ga^{(n, j)}} - R_{\gc, \gb^{(n, j)}} \|_{X(\R)} \geq & \| R_{\gc, \ga^{(n, j)}} - R_{\gc, \gb^{(n, j)}} \|_{X((- L_j/2, + \infty))}\\
\geq & \| Q_{c_N} \|_{X(\R)} - K \sum_{k = 1}^N \frac{2 - c_k^2}{c_k} \bigg( \int_{|x| \geq L_j/2} \exp \big( - 2 (2 - c_k^2)^\frac{1}{2} |x| \big) \, dx \bigg)^\frac{1}{2},
\end{align*}
for $n$ large enough. In the limit $n \to + \infty$, we deduce combining with \eqref{thion1} that
$$\| Q_{c_N} \|_{X(\R)} \leq K \sum_{k = 1}^N \frac{(2 - c_k^2)^\frac{3}{4}}{c_k} \exp \big( - (2 - c_k^2)^\frac{1}{2} L_j \big),$$
which is again not possible when $j \to + \infty$. Hence, there exists an integer $J$ such that, given any $j \geq J$, the sequence $(a_N^{(n, j)})_{n \in \N}$ is bounded. Up to some subsequence, it converges to a non-positive number $\alpha^j$. 

Assume next that $\alpha^j \leq - L_j/2$. For $n$ large enough, we can rely on \eqref{est:etavc} to write
\begin{align*}
\sum_{\ell = 0}^1 \big| \partial_x^\ell \eta_{\gc, \ga^{(n, j)}}(x) & - \partial_x^\ell \eta_{\gc, \gb^{(n, j)}}(x) + \partial_x^\ell \eta_{c_N}(x) \big|\\
+ \big| v_{\gc, \ga^{(n, j)}}(x) & - v_{\gc, \gb^{(n, j)}}(x) + v_{c_N}(x) \big| \leq K \sum_{k = 1}^N \frac{2 - c_k^2}{c_k} \exp \Big( - (2 - c_k^2)^\frac{1}{2} \Big( x - \frac{3 L_j}{8} \Big) \Big),
\end{align*}
for any $x \geq - L_j/8$. As above, this gives
$$\| Q_{c_N} \|_{X(\R)} \leq \| R_{\gc, \ga^{(n, j)}} - R_{\gc, \gb^{(n, j)}} \|_{X(\R)} + K \sum_{k = 1}^N \frac{(2 - c_k^2)^\frac{3}{4}}{c_k} \exp \Big( - \frac{(2 - c_k^2)^\frac{1}{2} L_j}{2} \Big),$$
which is not possible at the limit $n \to + \infty$, when $L_j$ is large enough. Therefore, we can assume, up to a further choice of $J$, that $\alpha^j \geq - L_j/2$ for any $j \geq J$.

Invoking once again \eqref{est:etavc}, we next write
\begin{align*}
\sum_{\ell = 0}^1 & \big| \partial_x^\ell \eta_{\gc, \ga^{(n, j)}}(x) - \partial_x^\ell \eta_{\gc, \gb^{(n, j)}}(x) - \partial_x^\ell \eta_{c_N, a_N^{(n, j)}}(x) + \partial_x^\ell \eta_{c_N}(x) \big|\\
+ & \big| v_{\gc, \ga^{(n, j)}}(x) - v_{\gc, \gb^{(n, j)}}(x) - v_{c_N, a_N^{(n, j)}}(x) + v_{c_N}(x) \big| \leq K \sum_{k = 1}^{N - 1} \frac{2 - c_k^2}{c_k} \exp \big( - (2 - c_k^2)^\frac{1}{2} (x - L_j) \big),
\end{align*}
for any $x \geq - 3 L_j/4$, so that
\begin{equation}
\label{millo}
\big\| Q_{c_N, a_N^{(n, j)}} - Q_{c_N} \big\|_{X(\R)} \leq \big\| R_{\gc, \ga^{(n, j)}} - R_{\gc, \gb^{(n, j)}} \big\|_{X(\R)} + K \sum_{k = 1}^{N - 1} \frac{(2 - c_k^2)^\frac{3}{4}}{c_k} \exp \Big( - \frac{(2 - c_k^2)^\frac{1}{2} L_j}{2} \Big).
\end{equation}
The expression in the right-hand side of \eqref{millo} tends to $0$ when $n \to + \infty$ and $L_j \to + \infty$. Up to a further choice of $J$, one can invoke Step \ref{D1} to show that
\begin{equation}
\label{chluski}
\big| a_N^{(j, n)} \big| \leq \frac{\delta}{4},
\end{equation}
for any $j \geq J$ and any $n \geq n_j$. Moreover, one can rephrase \eqref{millo} as
\begin{equation}
\label{pierre}
\begin{split}
\Big\| \sum_{k = 1}^{N - 1} \big( Q_{c_k, a_k^{(n, j)}} & - Q_{c_k, b_k^{(n, j)}} \big) \Big\|_{X(\R)}\\
& \leq 2 \| R_{\gc, \ga^{(n, j)}} - R_{\gc, \gb^{(n, j)}} \|_{X(\R)} + K \sum_{k = 1}^N \frac{(2 - c_k^2)^\frac{3}{4}}{c_k} \exp \Big( - \frac{(2 - c_k^2)^\frac{1}{2} L_j}{2} \Big).
\end{split}
\end{equation}
Once again, the expression in the right-hand side of \eqref{pierre} tends to $0$ when $n \to + \infty$ and $L_j \to + \infty$. Since Lemma \ref{lem:diffR} is true for the integer $N - 1$, we deduce that
$$\sum_{k = 1}^{N - 1} \big| a_k^{(j, n)} - b_k^{(j, n)} \big| \leq \frac{\delta}{4},$$
for any $j$ and $n$ large enough. Combining with \eqref{chluski}, we obtain a contradiction with \eqref{thion2}. Lemma \ref{lem:diffR} follows by induction on the integer $N$.
\end{proof}

Estimate \eqref{est:etavc} and Lemma \ref{lem:diffR} are enough to consider the proof of Proposition \ref{prop:stadec}.

\begin{proof}[Proof of Proposition \ref{prop:stadec}]
The proof is reminiscent from \cite{MarMeTs2}. It relies on the quantified version of the implicit function theorem provided by Proposition \ref{prop:implicit}. We consider the map $\Xi$ defined by
$$\Xi(\eta, v, \sigma, \gb) := \Big( \langle \eps, \partial_x Q_{\sigma_1, b_1} \rangle_{L^2}, \ldots, \langle \eps, \partial_x Q_{\sigma_N, b_N} \rangle_{L^2}, P'(Q_{\sigma_1, b_1})(\eps), \ldots, P'(Q_{\sigma_N, b_N})(\eps) \Big).$$
In order to simplify the notation, we have set here $\eps := (\eta, v) - R_{\sigma, \gb}$. The map $\Xi$ is well-defined and smooth from $X(\R) \times \Adm_N \times \R^N$ to $\R^{2 N}$. Moreover, it fulfills the assumptions of Proposition \ref{prop:implicit} so that we can state

\setcounter{step}{0}
\begin{step}
\label{I1}
Let $0 < \tau < 1$ and set
$$\Adm_N(\tau) := \big\{ \gc \in (- \sqrt{2}, \sqrt{2})^N, \ {\rm s.t.} \ \mu_{\gc} > \tau \ {\rm and} \ \nu_{\gc} > \tau \big\}.$$
There exist positive numbers $\delta$, $\Lambda$ and $M$, depending only on $\tau$, such that, given any $(\gc, \ga) \in \Adm_N(\tau) \times \Pos_N(M)$, there exists a map $\gamma_{\gc, \ga} \in \boC^1(B(R_{\gc, \ga}, \delta), \Adm_N(\tau/2) \times \R^N)$ such that, given any $(\eta, v) \in B(R_{\gc, \ga}, \delta)$, $(\sigma, \gb) = \gamma_{\gc, \ga}(\eta, v)$ is the unique solution in $B((\gc, \ga), \Lambda \delta)$ of the equation
$$\Xi(\eta, v, \sigma, \gb) = 0.$$
Moreover, the map $\gamma_{\gc, \ga}$ is Lipschitz on $B(R_{\gc, \ga}, \delta)$, with Lipschitz constant at most $\Lambda$.
\end{step}

For $\gc \in \Adm_N(\tau)$ and $\ga \in \R^N$, we check that
$$\Xi(R_{\gc, \ga}, \gc, \ga) = 0.$$
Similarly, we compute
$$\Bigg\{ \begin{array}{ll}
\partial_{\sigma_j} \Xi_k(R_{\gc, \ga}, \gc, \ga) = - \langle \partial_c Q_{c_j, a_j}, \partial_x Q_{c_k, a_k} \rangle_{L^2},\\
\partial_{b_j} \Xi_k(R_{\gc, \ga}, \gc, \ga) = \langle \partial_x Q_{c_j, a_j}, \partial_x Q_{c_k, a_k} \rangle_{L^2},
\end{array}$$
as well as
$$\Bigg\{ \begin{array}{ll}
\partial_{\sigma_j} \Xi_{N + k}(R_{\gc, \ga}, \gc, \ga) = - P'(Q_{c_k, a_k})(\partial_c Q_{c_j, a_j}),\\
\partial_{b_j} \Xi_{N + k}(R_{\gc, \ga}, \gc, \ga) = P'(Q_{c_k, a_k})(\partial_x Q_{c_j, a_j}),
\end{array}$$
for any $1 \leq j, k \leq N$. When $j = k$, we infer from formulae \eqref{form:etavc} that
\begin{equation}
\label{barcella}
\partial_{\sigma_k} \Xi_k(R_{\gc, \ga}, \gc, \ga) = \partial_{a_k} \Xi_{N + k}(R_{\gc, \ga}, \gc, \ga) = 0,
\end{equation}
and
$$\partial_{a_k} \Xi_k(R_{\gc, \ga}, \gc, \ga) = \| \partial_x Q_{c_k} \|_{L^2}^2 = \frac{(2 - c_k^2)^\frac{3}{2}}{3} > 0,$$
whereas by \eqref{def:Pc},
$$\partial_{\sigma_k} \Xi_{N + k}(R_{\gc, \ga}, \gc, \ga) = - \frac{d}{dc} \big( P(Q_c) \big)_{|c = c_k} = - (2 - c_k^2)^\frac{1}{2} < 0.$$
In particular, the diagonal matrix $A_\gc$ with the same diagonal elements as $d_{\sigma, \gb} \Xi(R_{\gc, \ga}, \gc, \ga)$ is a continuous isomorphism from $\R^{2 N}$ to $\R^{2 N}$, with operator norm bounded from below by $\tau^3/3$.

On the other hand, when $j \neq k$, it follows from \eqref{est:etavc} that
\begin{align*}
\big| \langle \partial_c Q_{c_j, a_j}, \partial_x Q_{c_k, a_k} \rangle_{L^2} \big| + \big| \langle \partial_x Q_{c_j, a_j}, \partial_x Q_{c_k, a_k} & \rangle_{L^2} \big|\\
& + \big| P'(Q_{c_k, a_k})(\partial_c Q_{c_j, a_j}) \big| + \big| P'(Q_{c_k, a_k})(\partial_x Q_{c_j, a_j}) \big|\\
\leq K^2 (2 - c_j^2)^\frac{1}{2} (2 - c_k^2)^\frac{1}{2} \Big( 1 & + \frac{1}{c_j^3 c_k^3} \Big) \int_\R \exp \big( - \nu_{j, k} (|x - a_j| + |x - a_k|) \big) dx\\
= K^2 (2 - c_j^2)^\frac{1}{2} (2 - c_k^2)^\frac{1}{2} \Big( 1 & + \frac{1}{c_j^3 c_k^3} \Big) \Big( |a_j - a_k| + \frac{1}{\nu_{j, k}} \Big) \exp \big( - \nu_{j, k} |a_j - a_k| \big),
\end{align*}
where we have set $\nu_{j, k} := \min \{ (2 - c_j^2)^\frac{1}{2}, (2 - c_k^2)^\frac{1}{2} \}$. Assuming that $\ga \in \Pos_N(L)$ for some positive number $L$, there exists a positive constant $K_\tau$, depending only on $\tau$, such that
\begin{align*}
|\partial_{\sigma_j} \Xi_k(R_{\gc, \ga}, \gc, \ga)| & + |\partial_{a_j} \Xi_k(R_{\gc, \ga}, \gc, \ga)|\\
& + |\partial_{\sigma_j} \Xi_{N + k}(R_{\gc, \ga}, \gc, \ga)| + |\partial_{a_j} \Xi_{N + k}(R_{\gc, \ga}, \gc, \ga)| \leq K_\tau (L + 1) \exp \big( - \nu_\gc L \big).
\end{align*}
Combining with \eqref{barcella}, we deduce when $L$ is large enough, that
$$d_{\sigma, \gb} \Xi(R_{\gc, \ga}, \gc, \ga) := A_\gc \big( Id + H(\gc, \ga) \big),$$
where the operator norm of $H(\gc, \ga)$ is less than $1/2$.

We now turn to the differential
\begin{align*}
d_{\eta, v} \Xi(R_{\gc, \ga}, \gc, \ga)(f, g) := \Big( \langle (f, g), \partial_x Q_{c_1, a_1} \rangle_{L^2}, & \ldots, \langle (f, g), \partial_x Q_{c_N, a_N} \rangle_{L^2},\\
P'(Q_{c_1, a_1})(f, g), & \ldots, P'(Q_{c_N, a_N})(f, g) \Big).
\end{align*}
Since the operator norm of $A_\gc^{- 1}$ is bounded by $3/\tau^3$, we infer from \eqref{est:etavc} that $d_{\eta, v} \Xi(R_{\gc, \ga}, \gc, \ga)$ may be written as
$$d_{\eta, v} \Xi(R_{\gc, \ga}, \gc, \ga) = A_\gc T_{\gc, \ga},$$
where $T_{\gc, \ga}$ is a continuous linear mapping from $X(\R)$ to $\R^N$ with operator norm depending only on $\tau$.

Finally, again by \eqref{est:etavc}, the operator norm of the second order differential $d^2 \Xi(\eta, v, \sigma, \gb)$ is bounded by a constant $K_\tau$, depending only on $\tau$, when $(\eta, v, \sigma, \gb) \in X(\R) \times \Adm_N(\tau/2) \times \R^N$. It then remains to notice that Assumption $(iv)$ of Proposition \ref{prop:implicit} is satisfied when $U = \Adm_N(\tau)$ and $V = \Adm_N(\tau/2)$, and to apply Proposition \ref{prop:implicit} to the map $\Xi$, in order to obtain the statements in Step \ref{I1}.

\begin{step}
End of the proof.
\end{step}

Let $\tau = \min \{ \mu_\gc^*/2, \nu_\gc^*/2 \}$. We denote $\delta_1$, $\Lambda_1$ and $M_1$, the constants provided by Step \ref{I1}, and $\beta_1$, the number provided by Lemma \ref{lem:diffR} for $\delta = \Lambda_1 \delta_1/3$. We set $\alpha_1 := \min \{ \delta_1/3, \beta_1/4 \}$ and $L_1 := M_1$. When $(\eta, v) \in \boU_{\gc^*}(\alpha_1, L_1)$, there exists $\gb \in \Pos_N(L_1)$ such that $(\eta, v) \in B(R_{\gc^*, \gb}, 2 \alpha_1)$. Since $2 \alpha_1 \leq \delta_1$ and $L_1 = M_1$, we deduce from Step \ref{I1} that the numbers $\gc$ and $\ga$ given by
$$(\gc, \ga) = \gamma_{\gc^*, \gb}(\eta, v),$$ 
are well-defined, so that we can set $\gC(\eta, v) = \gc$ and $\gA(\eta, v) = \ga$. We claim that $\gc$ and $\ga$ do not depend on the choice of $\gb \in \Pos_N(L_1)$ such that $(\eta, v) \in B(R_{\gc^*, \gb}, 2 \alpha_1)$. As a consequence, the maps $\gC$ and $\gA$ are well-defined on $\boU_{\gc^*}(\alpha_1, L_1)$ with values in $\Adm_N(\tau/2) \subset \Adm_N$, respectively, $\R^N$.

Indeed, given another choice of $\gb_2 \in \Pos_N(L_1)$ such that $(\eta, v) \in B(R_{\gc^*, \gb_2}, 2 \alpha_1)$, we have
$$\| R_{\gc^*, \gb} - R_{\gc^*, \gb_2} \|_X < 4 \alpha_1 \leq \beta_1,$$
so that, by Lemma \ref{lem:diffR},
\begin{equation}
\label{parra}
\big| \gb - \gb_2 \big| < \frac{\Lambda_1 \delta_1}{3}.
\end{equation}
On the other hand, the map $\gamma_{\gc^*, \gb}$ is Lipschitz on $B(R_{\gc^*, \gb}, 2 \alpha_1)$, with Lipschitz constant at most $\Lambda_1$. Hence,
$$\big| (\gc, \ga) - (\gc^*, \gb) \big| = \big| \gamma_{\gc^*, \gb}(\eta, v) - \gamma_{\gc^*, \gb}(R_{\gc^*, \gb}) \big| \leq \Lambda_1 \alpha_1 \leq \frac{\Lambda_1 \delta_1}{3}.$$
Combining with \eqref{parra}, we obtain
$$\big| (\gc, \ga) - (\gc^*, \gb_2) \big| < \Lambda_1 \delta_1.$$
Since $\Xi(\eta, v, \gc, \ga) = 0$, we deduce from Step \ref{I1} that $(\gc, \ga) = \gamma_{\gc^*, \gb_2}(\eta, v)$, so that $\gc$ and $\ga$ do not depend on the choice of $\gb \in \Pos_N(L_1)$ such that $(\eta, v) \in B(R_{\gc^*, \gb}, 2 \alpha_1)$.

Concerning the smoothness of $\gC$ and $\gA$, we consider $(\mu, w) \in \boU_{\gc^*}(\alpha_1, L_1)$ such that
$$\| (\mu, w) - (\eta, v) \|_X < \alpha_1.$$
Given $\gb \in \Pos_N(L_1)$ such that $(\eta, v) \in B(R_{\gc^*, \gb}, \alpha_1)$, we deduce that
$$\| (\mu, w) - R_{\gc^*, \gb} \|_X < 2 \alpha_1,$$
so that, by definition,
$$\big( \gC(\eta, v), \gA(\eta, v) \big) = \gamma_{\gc^*, \gb}(\eta, v) \quad {\rm and} \quad (\gC(\mu, w), \gA(\mu, w)) = \gamma_{\gc^*, \gb}(\mu, w).$$
At this stage, recall that the map $\gamma_{\gc^*, \gb}$ is of class $\boC^1$ on $B(R_{\gc^*, \gb}, 2 \alpha_1)$. As a consequence, the functions $\gC$ and $\gA$ are also of class $\boC^1$ on $B((\eta, v), \alpha_1) \cap \boU_{\gc^*}(\alpha_1, L_1)$, therefore, on $\boU_{\gc^*}(\alpha_1, L_1)$.

We next consider the proof of \eqref{est:stamodul1}. When $(\eta, v) \in B(R_{\gc^*, \ga^*}, \alpha)$, with $\ga^* \in \Pos_N(L)$, $\alpha \leq \alpha_1$ and $L \geq L_1$, the values of $\gC(\eta, v)$ and $\gA(\eta, v)$ are equal to
$$(\gC(\eta, v), \gA(\eta, v)) = \gamma_{\gc^*, \ga^*}(\eta, v).$$
In view of the Lipschitz continuity on $B(R_{\gc^*, \ga^*}, \alpha)$ of the map $\gamma_{\gc^*, \ga^*}$, we infer that
\begin{equation}
\label{trinhduc}
|\gc - \gc^*| + |\ga - \ga^*| \leq \Lambda_1 \| (\eta, v) - R_{\gc^*, \ga^*} \|_X \leq \Lambda_1 \alpha.
\end{equation}
On the other hand, letting $\eps = (\eta, v) - R_{\gc, \ga}$, we have
\begin{equation}
\label{skrela}
\| \eps \|_X \leq \| (\eta, v) - R_{\gc^*, \ga^*} \|_X + \| R_{\gc^*, \ga^*} - R_{\gc, \ga} \|_X < \alpha + \| R_{\gc^*, \ga^*} - R_{\gc, \ga} \|_X.
\end{equation}
In view of formulae \eqref{form:etavc}, there exists a universal constant $K$ such that
\begin{align*}
\| R_{\gc^*, \ga^*} - R_{\gc, \ga} \|_X & \leq \int_0^1 \Big\| \partial_\sigma R_{\gc + t (\gc^* - \gc), \ga + t (\ga^* - \ga)} (\gc^* - \gc) + \partial_\ga^* R_{\gc + t (\gc^* - \gc), \ga + t (\ga^* - \ga)}(\ga^* - \ga) \Big\|_X dt\\
& \leq K \big| (\gc^*, \ga^*) - (\gc, \ga) \big|.
\end{align*}
Combining with \eqref{trinhduc} and \eqref{skrela}, we are led to
$$\| \eps \|_X < \big( 1 + K \Lambda_1 \big) \alpha.$$
In view of \eqref{trinhduc}, it is sufficient to set $K_1 := 1 + (K + 1) \Lambda_1$ in order to derive \eqref{est:stamodul1}.

Finally, conditions \eqref{cond:staorth} are direct consequences of the definitions of the maps $\gamma_{\gc^*, \gb}$. This completes the proof of Proposition \ref{prop:stadec}.
\end{proof}

Corollary \ref{cor:toutvabien} is a direct consequence of Proposition \ref{prop:stadec}.

\begin{proof}[Proof of Corollary \ref{cor:toutvabien}]
Let $(\eta, v) \in \boU_{\gc^*}(\alpha_1, L)$, with $L \geq L_1$, and set $(\gc, \ga) := (\gC(\eta, v), \gA(\eta, v))$ as in the proof of Proposition \ref{prop:stadec}. In view of \eqref{est:stamodul1}, one can decrease $\alpha_1$ so that
$$|\gc - \gc^*| \leq \min \Big\{ \frac{\mu_{\gc^*}}{2}, \frac{\nu_{\gc^*}^2}{8 \sqrt{2}} \Big\},$$
which is enough to obtain \eqref{eq:ouf1} and \eqref{eq:ouf2}.

One can next increase $L_1$ so that $L_1 > 2$, and decrease again $\alpha_1$ so that $K_1 \alpha_1 \leq 1/2$. In this case, we can deduce from \eqref{est:stamodul1} that
$$|\ga - \ga^*| \leq \frac{1}{2},$$
for any $\ga^* \in \Pos_N(L)$ such that
$$\| (\eta, v) - R_{\gc^*, \ga^*} \|_X \leq \alpha.$$
It follows that $\ga$ belongs to $\Pos_N(L - 1)$ as mentioned in \eqref{est:stamodul2}.

Finally, one can infer from formula \eqref{form:etavc} that
$$\min \Big\{ 1 - \eta_c(x), \ x \in \R \Big\} = 1 - \eta_c(0) = \frac{c^2}{2},$$
for any $c \in (- \sqrt{2}, \sqrt{2})$. Combining with the exponential decay of the functions $\eta_c$ provided by \eqref{est:etavc}, we deduce that we can enlarge again $L_1$ so that
\begin{equation}
\label{eq:minsumeta}
\min \Big\{ 1 - \eta_{\gc, \ga}(x), \ x \in \R \Big\} \geq \frac{\mu_{\gc^*}^2}{4},
\end{equation}
when $\gc$ satisfies \eqref{eq:ouf2} and $\ga \in \Pos_N(L_1 - 1)$. On the other hand, we can a last time decrease $\alpha_1$ so that, combining \eqref{est:stamodul1} with the Sobolev embedding theorem, we have
$$\| \eps_\eta \|_{L^\infty(\R)} \leq \frac{\mu_{\gc^*}^2}{8}.$$
Since $\eta = \eta_{\gc, \ga} + \eps_\eta$, this is sufficient to guarantee that $\eta$ satisfies \eqref{eq:ouf3}, so that the pair $(\eta, v)$ belongs to $\boN\boV(\R)$. This completes the proof of Corollary \ref{cor:toutvabien}.
\end{proof}

\subsection{Almost minimizing properties close to a sum of solitons}

The main goal of this subsection is to establish the almost minimizing properties close to a sum of solitons 
stated in Subsection \ref{sub:decomp} of the introduction. When establishing the estimates of $E$ and $P_k$ given by \eqref{eq:decE} and \eqref{eq:decp}, we will make use of the following elementary inequality to bound various interaction terms using \eqref{est:etavc}.

\begin{lemma}
\label{lem:identite}
Let $(a, b) \in \R^2$, with $a < b$, $(\nu_a, \nu_b) \in (0, + \infty)^2$, and set $y^\pm := \max \{ \pm y, 0 \}$. Then,
$$\Big\| \exp \big( - \nu_a (\cdot - a)^+ \big) \exp \big( - \nu_b (\cdot - b)^- \big) \Big\|_{L^p} \leq \Big( \frac{2}{p \min \{ \nu_a, \nu_b \}} + b - a \Big)^\frac{1}{p} \exp \big( - \min \{ \nu_a, \nu_b \} (b - a) \big),$$
for any $1 \leq p \leq + \infty$. 
\end{lemma}

We will also use the following pointwise estimates on the functions $\Phi_k$, $\Phi_{k, k + 1}$ and $\Psi_k$.

\begin{lemma}
\label{lem:estimpsik}
Let $1 \leq k \leq N$ and $x \in \R$. We have
$$\Phi_k(x) \leq \exp \Big( - 2 \tau \Big( |x - a_k| - \frac{L_1}{4} \Big)^+ \Big),$$
and
$$|1 - \Phi_k(x)| \leq \exp \Big( - 2 \tau \Big( x - a_k + \frac{L_1}{4} \Big)^+ \Big) + \exp \Big( - 2 \tau \Big( x - a_k - \frac{L_1}{4} \Big)^- \Big).$$
Similarly, for $1 \leq k \leq N - 1$,
\begin{align*}
& \Phi_{0, 1}(x) \leq \exp \Big( - 2 \tau \Big( x - a_1 + \frac{L_1}{4} \Big)^+ \Big),\\
& \Phi_{k, k + 1}(x) \leq \exp \Big( - 2 \tau \Big( x - a_k - \frac{L_1}{4} \Big)^- \Big) \exp \Big( - 2 \tau \Big( x - a_{k + 1} + \frac{L_1}{4} \Big)^+ \Big),\\
& \Phi_{N, N + 1}(x) \leq \exp \Big( - 2 \tau \Big( x - a_N - \frac{L_1}{4} \Big)^- \Big),
\end{align*}
while, for $2 \leq k \leq N - 1$,
\begin{align*}
& \Psi_1(x) - \Psi_2(x) \leq \exp \Big( - \frac{\nu_{\gc^*}}{8} \Big( x - \frac{a_1 + a_2}{2} \Big)^+ \Big),\\
& \Psi_k(x) - \Psi_{k + 1}(x) \leq \exp \Big( - \frac{\nu_{\gc^*}}{8} \Big( x - \frac{a_{k - 1} + a_k}{2} \Big)^- \Big) \exp \Big( - \frac{\nu_{\gc^*}}{8} \Big( x - \frac{a_{k - 1} + a_k}{2} \Big)^+ \Big),\\
& \Psi_N(x) - \Psi_{N + 1}(x) \leq \exp \Big( - \frac{\nu_{\gc^*}}{8} \Big( x - \frac{a_{N - 1} + a_N}{2} \Big)^+ \Big),
\end{align*}
and
\begin{align*}
& \big| 1 - \Psi_1(x) + \Psi_2(x) \big| \leq \exp \Big( - \frac{\nu_{\gc^*}}{8} \Big( x - \frac{a_1 + a_2}{2} \Big)^- \Big),\\
& \big| 1 - \Psi_k(x) + \Psi_{k + 1}(x) \big| \leq \exp \Big( - \frac{\nu_{\gc^*}}{8} \Big( x - \frac{a_{k - 1} + a_k}{2} \Big)^+ \Big) + \exp \Big( - \frac{\nu_{\gc^*}}{8} \Big( x - \frac{a_{k - 1} + a_k}{2} \Big)^- \Big),\\
& \big| 1 - \Psi_N(x) + \Psi_{N + 1}(x) \big| \leq \exp \Big( - \frac{\nu_{\gc^*}}{8} \Big( x - \frac{a_{N - 1} + a_N}{2} \Big)^+ \Big).
\end{align*}
\end{lemma}

\begin{proof}
The estimates in Lemma \ref{lem:estimpsik} follow from the inequality
$$\big| 1 - \sign(x) \th(x) \big| \leq 2 \exp(- 2 |x|),$$
which holds for any $x \in \R$.
\end{proof}

We are now in position to present the 

\begin{proof}[Proof of Proposition \ref{prop:inutile}]
We begin with \eqref{eq:decE}. We expand the energy $E$ according to the Taylor rule by 
$$E(\eta, v) = E(R_{\gc, \ga} + \eps) = E(R_{\gc, \ga}) + E'(R_{\gc, \ga})(\eps) + \frac{1}{2} E''(R_{\gc, \ga})(\eps, \eps) + \boR(\eta, v),
$$
where
\begin{align}
\label{pourE}
& E(R_{\gc, \ga}) = \int_\R \bigg( \frac{(\partial_x \eta_{\gc, \ga})^2}{8 (1 - \eta_{\gc, \ga})} + \frac{1}{2} (1 - \eta_{\gc, \ga}) v_{\gc, \ga}^2 +\frac{1}{4} \eta_{\gc, \ga}^2 \bigg),\\
\label{pourE'}
& E'(R_{\gc, \ga})(\eps) = \int_\R \bigg( \frac{\partial_x \eta_{\gc, \ga} \partial_x \eps_\eta}{4(1 - \eta_{\gc, \ga})} + \frac{(\partial_x \eta_{\gc, \ga})^2 \eps_\eta}{8(1 - \eta_{\gc, \ga})} + (1 - \eta_{\gc, \ga}) v_{\gc, \ga} \eps_v + \frac{1}{2} \big( \eta_{\gc, \ga} - v_{\gc, \ga}^2 \big) \eps_\eta \bigg),\\
\label{pourE''}
& E''(R_{\gc, \ga})(\eps, \eps) = \int_\R \bigg( \frac{(\partial_x \eps_\eta)^2}{4 (1 - \eta_{\gc, \ga})} + \frac{\partial_x \eta_{\gc, \ga} \eps_\eta \partial_x \eps_\eta}{2 (1 - \eta_{\gc, \ga})^2} + \frac{(\partial_x \eta_{\gc, \ga})^2 \eps_\eta^2}{4 (1 - \eta_{\gc, \ga})^3} - 2 v_{\gc, \ga} \eps_\eta \eps_v + \big( 1 - \eta_{\gc, \ga} \big) \eps_v^2 + \frac{\eps_\eta^2}{2} \bigg),
\end{align}
and
$$\boR(\eta, v) := \frac{1}{8} \int_\R \Big( \frac{(\partial_x \eta_{\gc, \ga})^2 \eps_\eta^3}{(1 - \eta) (1 - \eta_{\gc, \ga})^3} + \frac{2 \partial_x \eta_{\gc, \ga} \partial_x \eps_\eta \eps_\eta^2}{(1 - \eta) (1 - \eta_{\gc, \ga})^2} + \frac{(\partial_x \eps_\eta)^2 \eps_\eta}{(1 - \eta) (1 - \eta_{\gc, \ga})} - 4 \eps_\eta \eps_v^2 \Big).$$
Following Corollary \ref{cor:toutvabien}, we may assume that \eqref{eq:ouf3}, \eqref{est:stamodul2}, \eqref{eq:ouf1} and \eqref{eq:ouf2}, as well as estimate \eqref{eq:minsumeta}, hold. From \eqref{est:etavc} and \eqref{eq:ouf2}, we obtain
\begin{equation}
\label{federer}
|\eta_{c_k, a_k}(x)| + |\partial_x \eta_{c_k, a_k}(x)| + |v_{c_k, a_k}(x)| = \boO \Big( \exp \Big( - \frac{\nu_{\gc^*} |x - a_k|}{2} \Big) \Big), 
\end{equation}
for any $x \in \R$ and any $1 \leq k \leq N$. From \eqref{eq:ouf3}, \eqref{eq:minsumeta}, \eqref{federer}, and the expression \eqref{pourE} for $E(R_{\gc, \ga})$, we compute using Lemma \ref{lem:identite} with $p = 1$,
\begin{equation}
\label{eq:decompE}
E(R_{\gc, \ga}) = \sum_{k = 1}^N E(Q_{c_k}) + \boO \Big( L \exp \Big( - \frac{\nu_{\gc^*} L}{2} \Big) \Big).
\end{equation}
Similarly, expression \eqref{pourE'} for $E'$ may be simplified using the H\"older inequality and Lemma \ref{lem:identite} with $p = 2$ as
$$E'(R_{\gc, \ga})(\eps) = \sum_{k=1}^N E'(Q_{c_k,a_k})(\eps) + \boO \Big( L^\frac{1}{2} \exp \Big( - \frac{\nu_{\gc^* } L}{2} \Big) \| \eps \|_X \Big).$$
Since $\| \eps \|_X \leq K_1 \alpha_1$ by \eqref{est:stamodul1} and since
$$E'(Q_{c_k, a_k})(\eps) = E'(Q_{c_k, a_k})(\eps) + c_k P'(Q_{c_k, a_k})(\eps) = 0,$$
by \eqref{eq:crit} and \eqref{cond:staorth}, we deduce that
\begin{equation}
\label{eq:decompE'}
E'(R_{\gc, \ga})(\eps) = \boO \Big( L^\frac{1}{2} \exp \Big( - \frac{\nu_{\gc^* } L}{2} \Big) \Big). 
\end{equation}
From \eqref{eq:ouf3}, \eqref{eq:minsumeta} and \eqref{federer}, as well as the Sobolev embedding theorem, we also infer that 
\begin{equation}
\label{eq:decompR}
\boR(\eta, v) = \boO \big( \| \eps \|_X^3 \big).
\end{equation}
We finally turn to the term $E''(R_{\gc, \ga})(\eps, \eps)$ which requires somewhat more care. We rewrite expression \eqref{pourE''} as
$$E''(R_{\gc, \ga})(\eps, \eps) := \int_\R J(R_{\gc, \ga}, \eps),$$
and we decompose according to the partition of unity \eqref{eq:partition} as
$$\int_\R J(R_{\gc, \ga}, \eps) = \sum_{k = 1}^N \int_\R J(R_{\gc, \ga}, \eps) \Phi_k + \sum_{k = 0}^N \int_\R J(R_{\gc, \ga}, \eps) \Phi_{k, k + 1}.$$
Invoking once more the decay estimates \eqref{federer}, we obtain from Lemma \ref{lem:identite} with $p = + \infty$, and Lemma \ref{lem:estimpsik},
\begin{equation}
\label{pioline}
\int_\R J(R_{\gc, \ga}, \eps) \Phi_k = \int_\R J(Q_{c_k, a_k}, \eps) \Phi_k + \boO \Big( \exp \Big( - \frac{\tau L_1}{2} \Big) \| \eps \|_X^2 \Big).
\end{equation}
Taking into account the formula
$$\partial_x \big( \Phi_k^\frac{1}{2} \eps_\eta \big) = \Phi_k^\frac{1}{2} \partial_x \eps_\eta + \frac{\partial_x \Phi_k}{2 \Phi_k^\frac{1}{2}} \eps_\eta,$$
we obtain, after a change of variable corresponding to a translation by $a_k$,
\begin{equation}
\label{benneteau}
\begin{split}
\int_\R J(Q_{c_k, a_k}, & \eps) \Phi_k\\
= & \int_\R J(Q_{c_k}, \eps_k) - \int_\R \bigg( \frac{\partial_x \Phi_k \eps_\eta^2}{32 \Phi_k (1 - \eta_{c_k, a_k})} + \frac{\eps_\eta \partial_x \eps_\eta}{8(1 - \eta_{c_k, a_k})} + \frac{\partial_x \eta_{c_k, a_k} \eps_\eta^2}{8(1 - \eta_{c_k, a_k})^2} \bigg) \partial_x \Phi_k.
\end{split}
\end{equation}
Since $|\partial_x \Phi_k| = \boO(\tau)$ and $|\partial_x \Phi_k|/|\Phi_k| = \boO(\tau)$, we deduce from \eqref{benneteau},
\begin{equation}
\label{eq:montre1}
\int_\R J(Q_{c_k, a_k}, \eps) \Phi_k = E''(Q_{c_k})(\eps_k, \eps_k) + \boO \big( \tau \| \eps \|_X^2 \big). 
\end{equation}
In a similar manner, we write using Lemma \ref{lem:identite} with $p = + \infty$, and Lemma \ref{lem:estimpsik},
$$\int_\R J(R_{\gc, \ga}, \eps) \Phi_{k, k + 1} = \int_\R J(0, \eps) \Phi_{k, k + 1} + \boO \Big( \exp \Big( - \frac{\tau L_1}{2} \Big) \| \eps \|_X^2 \Big),$$
and then
$$\int_\R J(0, \eps) \Phi_{k, k + 1} = \int_\R J(0, \eps_{k, k + 1}) + \boO \big( \tau \| \eps \|_X^2 \big) = E''(0)(\eps_{k, k + 1}, \eps_{k, k + 1}) + \boO \big( \tau \| \eps \|_X^2 \big).$$
Estimate \eqref{eq:decE} follows combining with \eqref{eq:decompE}, \eqref{eq:decompE'}, \eqref{eq:decompR}, \eqref{pioline} and \eqref{eq:montre1}. 

We turn now to \eqref{eq:decp}. We expand the functional $P_k$, which is exactly quadratic, as 
$$P_k(\eta, v) = P_k(R_{\gc, \ga} + \eps) = P_k(R_{\gc, \ga}) + P_k'(R_{\gc, \ga})(\eps) + \frac{1}{2} P_k''(R_{\gc, \ga})(\eps, \eps),$$
where
\begin{align}
\label{pourP}
& P_k(R_{\gc, \ga}) = \frac{1}{2} \int_\R \eta_{\gc, \ga} v_{\gc, \ga} \big( \Psi_k - \Psi_{k + 1} \big),\\
\label{pourP'}
& P_k'(R_{\gc, \ga})(\eps) = \frac{1}{2} \int_\R \big(\eta_{\gc, \ga} \eps_v + v_{\gc,\ga} \eps_\eta \big) \big( \Psi_k - \Psi_{k + 1} \big),\\
\label{pourP''}
& P_k''(R_{\gc, \ga})(\eps, \eps) = \int_\R \eps_\eta \eps_v \big( \Psi_k - \Psi_{k + 1} \big).
\end{align}
Invoking here again \eqref{federer}, Lemma \ref{lem:identite} with $p = 1$, and Lemma \ref{lem:estimpsik}, we estimate \eqref{pourP} as
$$P_k(R_{\gc, \ga}) = \sum_{j = 1}^N P_k(Q_{c_j, a_j}) + \boO \Big( L \exp \Big( - \frac{\nu_{\gc^*} L}{2} \Big) \Big).$$
For $j \neq k$, we obtain from Lemma \ref{lem:identite} with $p = 1$, and Lemma \ref{lem:estimpsik},
$$P_k(Q_{c_j, a_j}) = \boO \Big( L \exp \Big( - \frac{\nu_{\gc^*} L}{16} \Big) \Big),$$
while
$$P_k(Q_{c_k, a_k}) = P(Q_{c_k}) + \boO \Big( L \exp \Big( - \frac{\nu_{\gc^*} L}{16} \Big) \Big).$$
Concerning \eqref{pourP'}, we obtain similarly from Lemma \ref{lem:identite} with $p = 2$, and Lemma \ref{lem:estimpsik},
$$P_k'(R_{\gc, \ga})(\eps) = P'(Q_{c_k, a_k})(\eps) + \boO \Big( L^\frac{1}{2} \exp \Big( - \frac{\nu_{\gc^*} L}{16} \Big) \Big) = \boO \Big( L^\frac{1}{2} \exp \Big( - \frac{\nu_{\gc^*} L}{16} \Big) \Big),$$
where we have used \eqref{cond:staorth} for the last equality.

Finally, we decompose the integral in \eqref{pourP''} as
$$\int_\R \eps_\eta \eps_v \big( \Psi_k - \Psi_{k + 1} \big) = \sum_{j = 1}^N \int_\R \eps_\eta \eps_v \big( \Psi_k - \Psi_{k + 1} \big) \Phi_j + \sum_{j = 0}^N \int_\R \eps_\eta \eps_v \big( \Psi_k - \Psi_{k + 1} \big) \Phi_{j, j + 1}.$$
When $j \neq k$, we obtain
$$\int_\R \eps_\eta \eps_v \big( \Psi_k - \Psi_{k + 1} \big) \Phi_j = \boO \Big( \exp \Big( - \frac{\tau L_1}{2} \Big) \| \eps \|_X^2 \Big),$$
while for $j = k$,
\begin{align*}
\int_\R \eps_\eta \eps_v \big( \Psi_k - \Psi_{k + 1} \big) \Phi_k & = \int_\R \big( \eps_\eta \Phi_k^\frac{1}{2} \big) \big(\eps_v \Phi_k^\frac{1}{2} \big) + \int_\R \eps_\eta \eps_v \Phi_k \big( \Psi_k - \Psi_{k + 1} - 1 \big)\\
& = P''(Q_{c_k})(\eps_k, \eps_k) + \boO \Big( \exp \Big( - \frac{\tau L_1}{2} \Big) \| \eps \|_X^2 \Big).
\end{align*}
Also, when $j \neq k$ and $j \neq k - 1$, we obtain
$$\int_\R \eps_\eta \eps_v \big( \Psi_k - \Psi_{k + 1} \big) \Phi_{j, j + 1} = \boO \Big( \exp \Big( - \frac{\tau L_1}{2} \Big) \| \eps \|_X^2 \Big).$$
The remaining two terms (corresponding to $j = k$ and $j = k - 1$) are simply grouped together as
$$\int_\R \eps_\eta \eps_v \big( \Psi_k - \Psi_{k + 1} \big) \big( \Phi_{k - 1, k} + \Phi_{k, k + 1} \big) = P_k''(0)(\eps_{k - 1, k}, \eps_{k - 1, k}) + P_k''(0)(\eps_{k, k + 1}, \eps_{k, k + 1}).$$ 
The proof of estimate \eqref{eq:decp}, and therefore of Proposition \ref{prop:inutile}, is completed. 
\end{proof}

As a consequence of Proposition \ref{prop:inutile}, we derive

\begin{proof}[Proof of Corollary \ref{cor:pourG}]
We write
$$E(Q_{c_k}) = E(Q_{c_k^*}) + E'(Q_{c_k^*})( Q_{c_k} - Q_{c_k^*}) + \boO \big( |c_k - c_k^*|^2 \big),$$
and
$$P(Q_{c_k}) = P(Q_{c_k^*}) + P'(Q_{c_k^*})(Q_{c_k} - Q_{c_k^*}) + \boO \big( |c_k - c_k^*|^2 \big).$$
In view of Proposition \ref{prop:inutile}, the conclusion follows from the identity
$$E'(Q_{c_k^*}) + c_k^* P'(Q_{c_k^*}) = 0,$$
and from the inequalities (in order to recover the operators $H_{c_k}$ and $H_0^k$)
$$\big| (c_k^* - c_k) P''(Q_{c_k})(\eps_k, \eps_k) \big| \leq |c_k^* - c_k| \int_\R |\eps_\eta| |\eps_v| \leq \boO \big( |c_k - c_k^*|^2 \big) + \boO \big( \| \eps \|_X^4 \big),$$
and
$$\big| (c_k^* - c_k) P_k''(0)(\eps_{j, j + 1}, \eps_{j, j + 1}) \big| \leq |c_k^* - c_k| \int_\R |\eps_\eta| |\eps_v| \leq \boO \big( |c_k - c_k^*|^2 \big) + \boO \big( \| \eps \|_X^4 \big).$$
\end{proof}

We next prove Lemma \ref{lem:coerapp}.

\begin{proof}[Proof of Lemma \ref{lem:coerapp}]
Let $1 \leq k \leq N$. Under the assumptions of Corollary \ref{cor:pourG}, the pair $\eps$ satisfies the orthogonality conditions \eqref{cond:staorth}. In particular, we can write
$$\langle \eps_k, \partial_x Q_{c_k} \rangle_{L^2} = \int_\R \langle \eps, \partial_x Q_{c_k, a_k} \rangle_{\R^2} \Phi_k = \int_\R \langle \eps, \partial_x Q_{c_k, a_k} \rangle_{\R^2} \big( \Phi_k - 1 \big),$$
and
$$P'(Q_{c_k})(\eps_k) = \frac{1}{2} \int_\R \langle \eps, (v_{c_k, a_k}, \eta_{c_k, a_k}) \rangle_{\R^2} \Phi_k = \frac{1}{2} \int_\R \langle \eps, (v_{c_k, a_k}, \eta_{c_k, a_k}) \rangle_{\R^2} \big( \Phi_k - 1 \big).$$
Arguing as in the proof of Proposition \ref{prop:inutile}, we infer that
$$\int_\R |\eps| \big( |\partial_x Q_{c_k, a_k}| + |(v_{c_k, a_k}, \eta_{c_k, a_k})| \big) \big| \Phi_k - 1 \big| = \boO \Big( L_1^\frac{1}{2} \exp \Big( - \frac{\tau L_1}{2} \Big) \| \eps \|_{L^2} \Big),$$
so that
$$\langle \eps_k, \partial_x Q_{c_k} \rangle_{L^2} = \boO \Big( L_1^\frac{1}{2} \exp \Big( - \frac{\tau L_1}{2} \Big) \| \eps \|_{L^2} \Big),$$
and
\begin{equation}
\label{mcenroe}
P'(Q_{c_k})(\eps_k) = \boO \Big( L_1^\frac{1}{2} \exp \Big( - \frac{\tau L_1}{2} \Big) \| \eps \|_{L^2} \Big).
\end{equation}
We next write the following orthogonal decomposition of $\eps_k$,
\begin{equation}
\label{borg}
\eps_k := \frac{\langle \eps_k, \partial_x Q_{c_k} \rangle_{L^2}}{\| \partial_x Q_{c_k} \|_{L^2}} u_k + \frac{\langle \eps_k, (v_{c_k}, \eta_{c_k}) \rangle_{L^2}}{\| (v_{c_k}, \eta_{c_k}) \|_{L^2}} v_k + r_k,
\end{equation}
in which $u_k := \partial_x Q_{c_k}/\| \partial_x Q_{c_k} \|_{L^2}$, $v_k := (v_{c_k}, \eta_{c_k})/\| (v_{c_k}, \eta_{c_k}) \|_{L^2}$, and $r_k$ satisfies the orthogonality conditions \eqref{cond:orthc}. In view of Proposition \ref{prop:linstab}, we have
$$H_{c_k}(r_k) \geq \Lambda_{c_k} \| r_k \|_X^2.$$
On the other hand, since $\partial_x Q_{c_k}$ and hence $u_k$ belongs to the kernel of the quadratic form $H_{c_k}$, we can write 
$$H_{c_k}(\eps_k) = H_{c_k}(r_k) + 2 \frac{P'(Q_{c_k})(\eps_k)}{\| (v_{c_k}, \eta_{c_k}) \|_{L^2}} \langle \boH_{c_k}(v_k), \eps_k \rangle_{L^2} + \frac{\big( P'(Q_{c_k})(\eps_k) \big)^2}{\| (v_{c_k}, \eta_{c_k}) \|_{L^2}^2} H_{c_k}(v_k).$$
It follows from \eqref{mcenroe} that
\begin{equation}
\label{connors}
H_{c_k}(\eps_k) \geq \Lambda_{c_k} \| r_k \|_X^2 + \boO \Big( L_1^\frac{1}{2} \exp \Big( - \frac{\tau L_1}{2} \Big) \| \eps \|_{L^2}^2 \Big).
\end{equation}
Next, we similarly deduce from \eqref{borg} that
$$\big| \| r_k \|_X - \| \eps_k \|_X \big| \leq \bigg| \frac{\langle \eps_k, \partial_x Q_{c_k} \rangle_{L^2}}{\| \partial_x Q_{c_k} \|_{L^2}} \bigg| \| u_k \|_X + \bigg| \frac{P'(Q_{c_k})(\eps_k)}{\| (v_{c_k}, \eta_{c_k}) \|_{L^2}} \bigg| \| v_k \|_X = \boO \Big( L_1^\frac{1}{2} \exp \Big( - \frac{\tau L_1}{2} \Big) \|\eps\|_{L^2} \Big),$$
and therefore from \eqref{connors},
\begin{equation}
\label{edberg}
H_{c_k}(\eps_k) \geq \frac{\Lambda_{c_k}}{2} \| \eps_k \|_X^2 + \boO \Big( L_1^\frac{1}{2} \exp \Big( - \frac{\tau L_1}{2} \Big) \|\eps\|_{X}^2 \Big).
\end{equation}

We now turn to the terms $H_0^k(\eps_{k, k + 1})$. For $1 \leq k \leq N - 1$, we write
\begin{align*}
\big[ c_k P_k''(0) + c_{k + 1} P_{k + 1}''(0) \big] & (\eps_{k, k + 1}, \eps_{k, k + 1})\\
& = \int_\R c_k \big( \Psi_k - \Psi_{k + 1} \big) + c_{k + 1} \big( \Psi_{k + 1} - \Psi_{k + 2} \big) \big( \eps_{k, k + 1} \big)_\eta \big( \eps_{k, k + 1} \big)_v.
\end{align*}
Since
$$\big| c_k (\Psi_k - \Psi_{k + 1}) + c_{k + 1} (\Psi_{k + 1} - \Psi_{k + 2}) \big| \leq \max \big\{ |c_k|, |c_{k + 1}| \big\},$$
we obtain
$$H_0^k(\eps_{k, k + 1}) \geq \big[ E''(0) - \max \big\{ |c_k|, |c_{k + 1}| \big\} P''(0) \big] (\eps_{k, k + 1}, \eps_{k, k + 1}).$$
On the other hand, the quadratic form $E''(0) - \sqrt{2} P''(0)$ is non-negative, so that
\begin{equation}
\label{eq:sucre}
H_0^k(\eps_{k, k + 1}) \geq \Big( 1 - \frac{\max\{|c_k|, |c_{k + 1}|\}}{\sqrt{2}} \Big) E''(0)(\eps_{k, k + 1}, \eps_{k, k + 1}) \geq \frac{1}{4} \Big( 1 - \frac{\max\{|c_k|, |c_{k + 1}|\}}{\sqrt{2}} \Big) \| \eps_{k, k + 1} \|_X^2.
\end{equation}
Similarly,
\begin{equation}
\label{eq:sucrebis}
H_0^0(\eps_{0, 1}) \geq \Big( 1 - \frac{|c_1|}{\sqrt{2}} \Big) E''(0)(\eps_{0, 1}, \eps_{0, 1}) \geq \frac{1}{4} \Big( 1 - \frac{|c_{1}|}{\sqrt{2}} \Big) \| \eps_{0, 1} \|_X^2,
\end{equation}
and
\begin{equation}
\label{eq:sucreter}
H_0^N(\eps_{N, N + 1}) \geq \Big( 1 - \frac{|c_N|}{\sqrt{2}} \Big) E''(0)(\eps_{N, N + 1},\eps_{N, N + 1}) \geq \frac{1}{4} \Big( 1 - \frac{|c_N|}{\sqrt{2}} \Big) \| \eps_{N, N + 1} \|_X^2.
\end{equation}
In view of \eqref{edberg}, \eqref{eq:sucre}, \eqref{eq:sucrebis} and \eqref{eq:sucreter}, we set
$$\Lambda^* = \min \Big\{ \frac{\Lambda_{c_k}}{2}, \frac{1}{4} \Big( 1 - \frac{|c_k|}{\sqrt{2}} \Big), \ 1 \leq k \leq N \Big\},$$
so that \eqref{eq:miniepsk} and \eqref{eq:miniepskbis} follow. Notice that Corollary \ref{cor:toutvabien} and the last statement of Proposition \ref{prop:linstab} ensure that $\Lambda^*$ only depends on $\gc^*$.

In order to complete the proof of Lemma \ref{lem:coerapp}, we now derive \eqref{eq:equivnorm}. In view of definitions \eqref{def:epskbis} and \eqref{def:epsk}, we may write
$$\sum_{k = 1}^N \| \eps_k \|_X^2 = \sum_{k = 1}^N \int_\R \Big( \big( \partial_x (\eps_\eta \Phi_k^\frac{1}{2}) \big)^2 + \Phi_k \big( \eps_\eta^2 + \eps_v^2 \big) \Big),$$
and
$$\sum_{k = 0}^N \| \eps_{k, k + 1} \|_X^2 = \sum_{k = 0}^N \int_\R \Big( \big( \partial_x (\eps_\eta \Phi_{k, k + 1}^\frac{1}{2}) \big)^2 + \Phi_{k, k + 1} \big( \eps_\eta^2 + \eps_v^2 \big) \Big).$$
Expanding of the derivatives $\partial_x (\eps_\eta \Phi_k^\frac{1}{2})$ and $\partial_x (\eps_\eta \Phi_{k, k + 1}^\frac{1}{2})$ and summing, we are led to
\begin{align*}
\sum_{k = 1}^N \| \eps_k \|_X^2 & + \sum_{k = 0}^N \| \eps_{k, k + 1} \|_X^2 = \int_\R \Big( \sum_{k = 1}^N \Phi_k + \sum_{k = 0}^N \Phi_{k, k + 1} \Big) \big( (\partial_x \eps_\eta)^2 + \eps_\eta^2 + \eps_v^2 \big)\\
& + \int_\R \partial_x \Big( \sum_{k = 1}^N \Phi_k + \sum_{k = 0}^N \Phi_{k, k + 1} \Big) \eps_\eta \partial_x \eps_\eta + \sum_{k = 1}^N \int_\R \frac{(\partial_x \Phi_k)^2}{4 \Phi_k} \eps_\eta^2 + \sum_{k = 0}^N \int_\R \frac{(\partial_x \Phi_{k, k + 1})^2}{4 \Phi_{k, k + 1}} \eps_\eta^2.
\end{align*}
Since
$$\sum_{k = 1}^N \Phi_k + \sum_{k = 0}^N \Phi_{k, k + 1} = 1,$$
and $\Phi_k > 0$ on $\R$ (so that the last two sums in the previous identity are non-negative), we finally deduce that
$$\sum_{k = 1}^N \| \eps_k \|_X^2 + \sum_{k = 0}^N \| \eps_{k, k + 1} \|_X^2 \geq \int_\R \big( (\partial_x \eps_\eta)^2 + \eps_\eta^2 + \eps_v^2 \big) = \| \eps \|_X^2.$$
This completes the proof of Lemma \ref{lem:coerapp}.
\end{proof}

\begin{proof}[Proof of Corollary \ref{cor:estG}]
Inequality \eqref{eq:lowerG} follows from Corollary \ref{cor:pourG}, Lemma \ref{lem:coerapp} and inequality \eqref{eq:capasse}, while inequality \eqref{eq:upperG} follows from Corollary \ref{cor:pourG} and the fact that $H_{c_k}$ and $H_0$ are continuous quadratic forms, with continuous bounds depending only on $\gc^*$, in view of Corollary \ref{cor:toutvabien}.
\end{proof}

\section{Dynamical properties}
\label{sec:evol}

In this section, we focus on the evolution in time under \eqref{HGP} of various quantities introduced so far. In particular, we present the proofs of Propositions \ref{prop:modul} and \ref{prop:mono}. 

\subsection{Dynamics of the modulation parameters}

\begin{proof}[Proof of Proposition \ref{prop:modul}] 
We first derive the conclusions of Proposition \ref{prop:modul} when the pair $(\eta, v)$ is a smooth solution to \eqref{HGP}, for which we are allowed to compute the differential equations satisfied by the modulation parameters $a_k(t)$ and $c_k(t)$. We next complete the proof using a density argument.

\setcounter{step}{0}
\begin{step}
\label{S1}
Regularity and estimates for smooth solutions. 
\end{step}

In this step, we make the further assumption that $(\eta, v) \in \boC^0([0,T],X^3(\R))$, from which it follows by equation \eqref{HGP} that $(\eta, v) \in \boC^1([0,T],X(\R))$, and then from Proposition \ref{prop:stadec} that $\gc \in \boC^1([0,T], \Adm_N)$ and $\ga \in \boC^1([0, T], \R^N)$. We derive from \eqref{HGP} that $\eps := (\eps_\eta, \eps_v)$ is solution to the equations
\begin{equation}
\label{eq:epseta}
\begin{split}
\partial_t \eps_\eta = & 2 \partial_x \big( \eps_v - \eps_\eta v_{\gc, \ga} - \eps_v \eta_{\gc, \ga} - \eps_\eta \eps_v \big) - \sum_{k = 1}^N \sum_{j \neq k} \partial_x \big( \eta_{c_k, a_k} v_{c_j, a_j} \big)\\
& - \sum_{k = 1}^N c_k' \partial_c \eta_{c_k, a_k} + \sum_{k = 1}^N \big( a_k' - c_k \big) \partial_x \eta_{c_k, a_k},
\end{split} 
\end{equation}
and
\begin{equation}
\label{eq:epsv}
\begin{split}
\partial_t \eps_v = & \partial_x \big( \eps_\eta - 2 v_{\gc, \ga} \eps_v - (\eps_v)^2 \big) - \sum_{k = 1}^N \sum_{j \neq k} \partial_x \big( v_{c_k, a_k} v_{c_j, a_j} \big) + \sum_{k = 1}^N \partial_x^2 \Big( \frac{\partial_x \eta_{c_k, a_k}}{2 (1 - \eta_{c_k, a_k})} \Big)\\
& - \partial_x^2 \Big( \frac{\partial_x \eta_{\gc, \ga} + \partial_x \eps_\eta}{2 (1 - \eta_{\gc, \ga} - \eps_\eta)} \Big) + \sum_{k = 1}^N \partial_x \Big( \frac{3 (\partial_x \eta_{c_k, a_k})^2}{4 (1 - \eta_{c_k, a_k})^2} \Big) - \partial_x \Big( \frac{3 (\partial_x \eta_{\gc, \ga} + \partial_x \eps_\eta)^2}{4 (1 - \eta_{\gc, \ga} - \eps_\eta)^2} \Big)\\
& - \sum_{k = 1}^N c_k' \partial_c v_{c_k, a_k} + \sum_{k = 1}^N \big( a_k' - c_k \big) \partial_x v_{c_k, a_k}.
\end{split}
\end{equation}
Here and in the sequel, we have dropped the explicit mention to the dependence in $t$ when this does not lead to a confusion. 

We next differentiate the orthogonality conditions \eqref{cond:orthc} to obtain
$$\frac{d}{dt} \langle \eps(\cdot, t), \partial_x Q_{c_k(t), a_k(t)} \rangle_{L^2} = \frac{d}{dt} P'(Q_{c_k(t), a_k(t)})(\eps(\cdot, t)) = 0,$$
for any $1 \leq k \leq N$. Expanding the previous identities according to the chain rule and combining the resulting equations with \eqref{eq:epseta} and \eqref{eq:epsv} yields the system 
\begin{equation}
\label{lloris}
M(t) \begin{pmatrix} \gc'(t) \\ \ga'(t) - \gc(t) \end{pmatrix} = \Phi(t),
\end{equation}
where the coefficients $(M_{k, j}(t))_{1 \leq j, k \leq 2 N}$ of the $2 N$-dimensional matrix $M(t)$ are given by
\begin{align*}
& M_{k, j}(t) := P'(Q_{c_k, a_k})(\partial_c Q_{c_j, a_j}) - \delta_{j, k} P'(\partial_c Q_{c_k, a_k})(\eps),\\
& M_{k, j + N}(t) := - P'(Q_{c_k, a_k})(\partial_x Q_{c_j, a_j}) + \delta_{j, k} P'(\partial_x Q_{c_k, a_k})(\eps),\\
& M_{k + N, j}(t) := \langle \partial_c Q_{c_j, a_j}, \partial_x Q_{c_k, a_k} \rangle_{L^2} - \delta_{j, k} \langle \eps, \partial_c \partial_x Q_{c_k, a_k} \rangle_{L^2},\\
& M_{k + N, j + N}(t) := - \langle \partial_x Q_{c_j, a_j}, \partial_x Q_{c_k, a_k} \rangle_{L^2} + \delta_{j, k} \langle \eps, \partial_x^2 Q_{c_k, a_k} \rangle_{L^2},
\end{align*}
and the components $(\Phi_k(t))_{1 \leq k \leq 2 N}$ of the $2 N$-vector $\Phi(t)$ are given by
\begin{equation}
\label{mexes1}
\begin{split}
\Phi_k(t) := & - 2 \langle \eps_v - \eps_\eta v_{\gc, \ga} - \eps_v \eta_{\gc, \ga} - \eps_\eta \eps_v, \partial_x v_{c_k, a_k} \rangle_{L^2} - \langle \eps_\eta - 2 v_{\gc, \ga} \eps_v - (\eps_v)^2, \partial_x \eta_{c_k, a_k} \rangle_{L^2}\\
& - c_k P'(\partial_x Q_{c_k, a_k})(\eps) - \Big\langle \frac{\partial_x \eta_{\gc, \ga} + \partial_x \eps_\eta}{2 (1 - \eta_{\gc, \ga} - \eps_\eta)} - \sum_{j = 1}^N \frac{\partial_x \eta_{c_j, a_j}}{2 (1 - \eta_{c_j, a_j})}, \partial_x^2 \eta_{c_k, a_k} \Big\rangle_{L^2}\\
& + \Big\langle \frac{3 (\partial_x \eta_{\gc, \ga} + \partial_x \eps_\eta)^2}{4 (1 - \eta_{\gc, \ga} - \eps_\eta)^2} - \sum_{j = 1}^N \frac{3 (\partial_x \eta_{c_j, a_j})^2}{4 (1 - \eta_{c_j, a_j})^2}, \partial_x \eta_{c_k, a_k} \Big\rangle_{L^2}\\
& + \sum_{j = 1}^N \sum_{\underset{l \neq j}{l = 1}}^N \Big( \langle \eta_{c_j, a_j} v_{c_l, a_l}, \partial_x v_{c_k, a_k} \rangle_{L^2} + \langle v_{c_j, a_j} v_{c_l, a_l}, \partial_x \eta_{c_k, a_k} \rangle_{L^2} \Big),
\end{split}
\end{equation}
\begin{equation}
\label{mexes2}
\begin{split}
\Phi_{k + N}(t) := & 2 \langle \eps_\eta v_{\gc, \ga} + \eps_v \eta_{\gc, \ga} + \eps_\eta \eps_v - \eps_v, \partial_x^2 \eta_{c_k, a_k} \rangle_{L^2} - \langle \eps_\eta - 2 v_{\gc, \ga} \eps_v - (\eps_v)^2, \partial_x^2 v_{c_k, a_k} \rangle_{L^2}\\
& - c_k \langle \eps, \partial_x^2 Q_{c_k, a_k} \rangle_{L^2} - \Big\langle \frac{\partial_x \eta_{\gc, \ga} + \partial_x \eps_\eta}{2 (1 - \eta_{\gc, \ga} - \eps_\eta)} - \sum_{j = 1}^N \frac{\partial_x \eta_{c_j, a_j}}{2 (1 - \eta_{c_j, a_j})}, \partial_x^3 v_{c_k, a_k} \Big\rangle_{L^2}\\
& + \Big\langle \frac{3 (\partial_x \eta_{\gc, \ga} + \partial_x \eps_\eta)^2}{4 (1 - \eta_{\gc, \ga} - \eps_\eta)^2} - \sum_{j = 1}^N \frac{3 (\partial_x \eta_{c_j, a_j})^2}{4 (1 - \eta_{c_j, a_j})^2}, \partial_x^2 v_{c_k, a_k} \Big\rangle_{L^2}\\
& + \sum_{j = 1}^N \sum_{\underset{l \neq j}{l = 1}}^N \Big( \langle \eta_{c_j, a_j} v_{c_l, a_l}, \partial_x^2 \eta_{c_k, a_k} \rangle_{L^2} + \langle v_{c_j, a_j} v_{c_l, a_l}, \partial_x^2 v_{c_k, a_k} \rangle_{L^2} \Big).
\end{split}
\end{equation}
We next write the matrix $M(t)$ under the form $M(t) := D(t) + H(t)$, where $D(t)$ refers to the diagonal matrix with diagonal coefficients
\begin{align*}
& D_{k, k}(t) := P'(Q_{c_k})(\partial_c Q_{c_k}) = (2 - c_k^2)^\frac{1}{2} \neq 0,\\
& D_{k + N, k + N}(t) := \| \partial_x Q_{c_k} \|_{L^2}^2 = \frac{(2 - c_k^2)^\frac{3}{2}}{3} \neq 0,
\end{align*}
where $1 \leq k \leq N$. In particular, $D(t)$ is invertible, and its inverse is uniformly bounded with respect to $t$ in view of \eqref{eq:ouf1}. In order to estimate the matrix $H(t)$, we first invoke the identities
$$P'(Q_{c_k, a_k})(\partial_x Q_{c_k, a_k}) = \langle \partial_c Q_{c_k, a_k}, \partial_x Q_{c_k, a_k} \rangle_{L^2} = 0.$$
Since the remaining terms involve either $\eps$ or distinct solitons, we can rely on computations similar to the ones in the proof of Proposition \ref{prop:inutile} to get
$$\big| H_{k, j}(t) \big| = \boO \big( \| \eps \|_{L^2} \big) + \boO \Big( L \exp \Big( - \frac{\nu_{\gc^*} L}{2} \Big) \Big).$$
Hence, there exists $\alpha_2 \leq \alpha_1$ and $L_2 \geq L_1$ such that for $\alpha \leq \alpha_2$ and $L \geq L_2$, the operator norm of the matrix $D(t)^{- 1} H(t)$ is less than $1/2$, so that $M(t)$ is invertible. Moreover, the operator norms of the matrices $M(t)^{- 1}$ are uniformly bounded with respect to $t$.

At this stage, we have proved that
$$\sum_{k = 1}^N \big( |a_k'(t) - c_k(t)| + |c_k'(t)| \big) = \boO \big( \| \Phi(t) \|_{\R^{2 N}} \big),$$
for any $t \in [0, T]$. Inspection of $\Phi(t)$ shows that, similar to $H_{k, j}$ above, it only contains terms involving either $\eps$ or distinct solitons. Therefore we obtain
$$\sum_{k = 1}^N \big( |a_k'(t) - c_k(t)| + |c_k'(t)| \big) = \boO \big( \| \eps \|_{L^2} \big) + \boO \Big( L \exp \Big( - \frac{\nu_{\gc^*} L}{2} \Big) \Big),$$ 
which corresponds to \eqref{est:modul2}.

\begin{step}
\label{S2}
Density argument. 
\end{step}

We consider a sequence $(\eta_0^n, v_0^n)$ in $\boN\boV^3(\R)$ such that
$$(\eta_0^n, v_0^n) \to (\eta_0, v_0) \quad {\rm in} \quad X(\R),$$
as $n \to + \infty$, and we denote $(\eta^n, v^n)$ the corresponding solutions to \eqref{HGP}. It follows from Theorem \ref{thm:cauchyhgp} that the solutions $(\eta^n, v^n)$ are well-defined on $[0, T]$ for $n$ sufficiently large and that
\begin{equation}
\label{benzema1}
(\eta^n, v^n) \to (\eta, v) \quad {\rm in} \quad \boC^0([0, T], X(\R)),
\end{equation}
as $n \to + \infty$. In particular, for $n$ large enough, the compactness of the segment $[0, T]$ ensures that $(\eta^n(\cdot, t), v^n(\cdot, t))$ belongs to $\boU_{\gc^*}(\alpha, L)$ for any $t \in [0, T]$. As a consequence, one can apply Step \ref{S1} to $(\eta^n, v^n)$. This yields maps $\gc^n$ and $\ga^n$ of class $\boC^1$ on $[0, T]$ such that the orthogonality conditions \eqref{cond:orthc} are satisfied by the pair $\eps^n := (\eps_\eta^n, \eps_v^n) = (\eta^n, v^n) - R_{\gc^n, \ga^n}$. Moreover, combining the Lipschitz continuity of the maps $\gC$ and $\gA$ on some tubular neighborhood of the compact set $\{ (\eta(t), v(t)), \ t \in [0, T] \}$ with \eqref{benzema1}, we have
\begin{equation}
\label{benzema2}
(\gc^n(t), \ga^n(t)) \to (\gc(t), \ga(t)) \quad {\rm in} \quad \boC^0([0, T], \R^{2 N}),
\end{equation}
as $n \to + \infty$. Recall that, by formulae \eqref{form:etavc}, the maps $(c, a) \to \partial_c^\alpha \partial_x^\beta Q_c(\cdot - a)$ are continuous from $\big( (- \sqrt{2}, \sqrt{2}) \setminus \{ 0 \} \big) \times \R$ into $L^2(\R)^2$ for any $(\alpha, \beta) \in \N^2$. Combined with \eqref{benzema2}, this shows that the matrices $M^n(t)$ converge towards the matrices $M(t)$, uniformly with respect to $t$, as $n \to + \infty$. Since the map $A \mapsto A^{- 1}$ is continuous from $\boG \boL_N(\R)$ to $\boM_N(\R)$, the inverses $(M^n(t))^{- 1}$ also converge towards the inverses $M(t)^{- 1}$ uniformly with respect to $t \in [0, T]$.

We next turn to the vectors $\Phi^n(t)$. In view of \eqref{benzema1} and \eqref{benzema2}, we can write
\begin{equation}
\label{mandanda}
\eps^n = (\eta^n, v^n) - R_{\gc^n, \ga^n} \to (\eta, v) - R_{\gc, \ga} = \eps \quad {\rm in} \quad \boC^0([0, T], X(\R)),
\end{equation}
as $n \to + \infty$. On the other hand, in view of formulae \eqref{form:etavc}, the derivatives $\partial_c^\alpha \partial_x^\beta Q_c$ have a vanishing limit at infinity for any $(\alpha, \beta) \in \N^2$. Therefore, the maps $(c, a) \to \partial_c^\alpha \partial_x^\beta Q_c(\cdot - a)$ are not only continuous for the $L^2$-norm, but also for the uniform one. Applying \eqref{benzema2} and \eqref{mandanda} to expressions \eqref{mexes1} and \eqref{mexes2}, and invoking the Sobolev embedding theorem of $H^1(\R)$ into $\boC^0(\R)$ when necessary, it follows that
$$\Phi^n \to \Phi \quad {\rm in} \quad \boC^0([0, T], \R^{2 N}),$$
as $n \to + \infty$. Coming back to \eqref{lloris}, we are finally led to
$$\begin{pmatrix} (\gc^n)'(t) \\ (\ga^n)'(t) - \gc^n(t) \end{pmatrix} \to M(t)^{- 1} \Phi(t),$$
as $n \to + \infty$, the convergence being uniform with respect to $t \in [0, T]$. In view of \eqref{benzema2}, this shows that the maps $\gc$ and $\ga$ are actually of class $\boC^1$ on $[0, T]$. Moreover, they are given by
$$\begin{pmatrix} \gc'(t) \\ \ga'(t) - \gc(t) \end{pmatrix} = M(t)^{- 1} \Phi(t),$$
for any $t \in [0, T]$, so that \eqref{est:modul2} follows as in Step \ref{S1}.
\end{proof}

We are now in position to provide the

\begin{proof}[Proof of Corollary \ref{cor:separated}]
When $(\eta(\cdot, t), v(\cdot, t)) \in \boU_{\gc^*}(\alpha, L)$ for any $t \in [0, T]$, we can combine estimates \eqref{est:stamodul1} and \eqref{est:modul2} to obtain
\begin{equation}
\label{lievremont}
\sum_{k = 1}^N \big| a_k'(t) - c_k^* \big| \leq \sum_{k = 1}^N \big| a_k'(t) - c_k(t) \big| + \sum_{k = 1}^N \big| c_k(t) - c_k^* \big| = \boO \big( \alpha \big) + \boO \Big( \Big( L \exp \Big( - \frac{\nu_{\gc^*} L}{2} \Big) \Big).
\end{equation}
As a consequence, we can fix $\alpha_3$ small enough and $L_3$ sufficiently large, so that
$$a_{k + 1}'(t) - a_k'(t) > \frac{1}{2} \big( c_{k + 1}^* - c_k^* \big) \geq \sigma^*,$$
for any $1 \leq k \leq N - 1$, when $\alpha \leq \alpha_3$ and $L \geq L_3$. The first inequality in \eqref{est:stamodul2bis} follows by writing
$$a_{k + 1}(t) - a_k(t) - a_{k + 1}(0) + a_k(0) = \int_0^t \big( a_{k + 1}'(s) - a_k(s) \big) \, ds > \sigma^* t.$$
Applying \eqref{est:stamodul2} to the pair $(\eta(\cdot, 0), v(\cdot, 0))$, we deduce the second inequality in \eqref{est:stamodul2bis}.

Concerning \eqref{def:Wtierce}, it follows from \eqref{lievremont} that
$$\Big| \sqrt{2 - (a_k'(t))^2} - \sqrt{2 - (c_k^*)^2} \Big| = \boO \big( \alpha \big) + \boO \Big( \Big( L \exp \Big( - \frac{\nu_{\gc^*} L}{2} \Big) \Big).$$
Therefore, one can decrease $\alpha_3$ and increase $L_3$, if necessary, so that inequality \eqref{def:Wtierce} holds for any $1 \leq k \leq N$, when $\alpha \leq \alpha_3$ and $L \geq L_3$. This completes the proof of Corollary \ref{cor:separated}.
\end{proof}

\subsection{Monotonicity of localized scalar momentum}
\label{sub:monotonicity}

This subsection is devoted to the almost monotonicity properties of suitably localized versions of the scalar momentum
$$P(\eta, v) := \frac{1}{2} \int_\R \eta v.$$
We derive them using the conservative form of the equation governing the integrand $\eta v$ of $P$.

\begin{lemma}
\label{lem:p}
Let $(\eta, v) \in \boC^0([0, T],\boN\boV^3(\R))$ be a solution to \eqref{HGP}. Then,
\begin{equation}
\label{conslaw}
\partial_t \big( \eta v \big) = \partial_x \Big( (1 - 2 \eta) v^2 + \frac{\eta^2}{2} + \frac{(3 - 2 \eta) (\partial_x \eta)^2}{4(1 - \eta)^2} \Big) + \frac{1}{2} \partial_x^3 \big( \eta + \ln(1 - \eta) \big) \quad {\rm on} \quad [0, T] \times \R.
\end{equation}
\end{lemma}

\begin{proof}
Under such regularity, it follows from \eqref{HGP} that $\partial_t \eta$ and $\partial_t v$ are continuous functions on $[0, T] \times \R$ and that 
$$\partial_t \big( \eta v \big) = \partial_x \Big( (1 - 2 \eta) v^2 + \frac{\eta^2}{2} \Big) - \eta \partial_x \Big( \frac{\partial_x^2 \eta}{2(1 - \eta)} + \frac{(\partial_x \eta)^2}{4(1 - \eta)^2} \Big).$$
The conservation law given by \eqref{conslaw} then follows from the identity
$$\eta \partial_x \Big( \frac{\partial_x^2 \eta}{2(1 - \eta)} + \frac{(\partial_x \eta)^2}{4(1 - \eta)^2} \Big) = - \frac{1}{2} \partial_x^3 \big( \eta + \ln(1 - \eta) \big) - \partial_x \Big( \frac{(3 - 2 \eta) (\partial_x \eta)^2}{4(1 - \eta)^2} \Big).$$
\end{proof}

One can drop the smoothness assumptions in Lemma \ref{lem:p} by deriving an integral version of \eqref{conslaw}.

\begin{cor}
\label{Cor:intp}
Let $(\eta, v) \in \boC^0([0, T], \boN\boV(\R))$ be a solution to \eqref{HGP}, and let $\psi \in \boC^0([0, T], \linebreak[1] \boC_b^3(\R)) \cap \boC^1([0, T],\boC_b^0(\R))$. Then, on $[0,T]$ 
\begin{equation}
\label{consint}
\partial_t \bigg( \int_\R \psi \eta v \bigg) = \int_\R \partial_t \psi \eta v - \int_\R \partial_x \psi \Big( (1 - 2 \eta) v^2 + \frac{\eta^2}{2} + \frac{(3 - 2 \eta) (\partial_x \eta)^2}{4(1 - \eta)^2} \Big) - \frac{1}{2} \int_\R \partial_x^3 \psi \big( \eta + \ln(1 - \eta) \big).
\end{equation}
\end{cor}

\begin{proof}
It follows by approximation of the initial datum using Theorem \ref{thm:cauchyhgp}.
\end{proof}

In particular, taking $\psi = 1$ in Corollary \ref{Cor:intp} yields the conservation of $P$.

We turn now to localized versions of $P$. We deduce from Corollary \ref{Cor:intp}

\begin{prop}
\label{prop:almomo}
Let $(\eta, v) \in \boC^0([0, T], \boN\boV(\R))$ be a solution to \eqref{HGP}. Let $\Psi \in \boC_b^3(\R)$ be a non-decreasing function such that $|\partial_x^3 \Psi| \leq C_0 \partial_x \Psi$ on $\R$ for some positive number $C_0$. Let $X \in \boC^1([0, T], \R)$, $t_0 \in [0, T]$ and $\tau_0 > 0$ be such that $2 C_0 \tau_0^2 < 1$, and
\begin{equation}
\label{eq:pastropgrand}
\eta(t_0, x) \leq \frac{1}{2} - \frac{X'(t_0)^2}{4(1 - 2 C_0 \tau_0^2)},
\end{equation}
for any $x \in [X(t_0) - R_0, X(t_0) + R_0]$ and some positive number $R_0$. Then,
\begin{equation}
\label{eq:almono}
\partial_t \bigg( \int_\R \Psi \big( \tau_0 (x - X(t)) \big) \eta(x, t) v(x, t) dx \bigg)_{|t = t_0} \leq \tau_0 \Big( 6 + \frac{6}{1 - \eta_{\rm sup}} \Big) E(\eta, v) \sup_{|x| \geq \tau_0 R_0} \partial_x \Psi(x),
\end{equation}
where
$$\eta_{\rm sup} := \sup \big\{ \eta(t, x), \ (t, x) \in [0, T] \times \R \big\} < 1.$$
\end{prop}

\begin{proof}
Set $I_0 := [X(t_0) - R_0, X(t_0) + R_0]$. It follows from \eqref{eq:pastropgrand} and the condition $2 C_0 \tau_0^2 < 1$ that $\eta \leq 1/2$ on $I_0$, and real analysis then implies that
\begin{equation}
\label{eq:comparelog}
|\eta + \ln(1 - \eta)| \leq \eta^2 \quad {\rm on} \quad I_0. 
\end{equation}
Applying \eqref{consint}, we may rewrite the left-hand side of \eqref{eq:almono} as
$$- \int_\R \tau_0 \partial_x \Psi \Big( (1 - 2 \eta) v^2 + \frac{\eta^2}{2} + X'(t_0) \eta v + \frac{(3 - 2 \eta) (\partial_x \eta)^2}{4(1 - \eta)^2} \Big) - \frac{1}{2} \int_\R \tau_0^3 \partial_x^3 \Psi \big( \eta + \ln(1 - \eta) \big),$$ 
where $\partial_x \Psi$ and $\partial_x^3 \Psi$ are evaluated at the point $\tau_0 (x - X(t_0))$.

For $x \in I_0$, we deduce from \eqref{eq:comparelog} and the bound $|\partial_x^3 \Psi| \leq C_0 \partial_x \Psi$ that
\begin{align*}
\tau_0 \partial_x \Psi \Big( (1 - 2 \eta) v^2 + \frac{\eta^2}{2} + X'(t_0) \eta v \Big) & + \frac{1}{2} \tau_0^3 \partial_x^3 \Psi \big( \eta + \ln(1 - \eta) \big)\\
& \geq \tau_0 \partial_x \Psi \Big( (1 - 2 \eta) v^2 + \Big( \frac{1}{2} - C_0 \tau_0^2 \Big) \eta^2 - |X'(t_0)| \eta v \Big) \geq 0,
\end{align*}
where the non-negativity of the quadratic form, which yields the last inequality, follows from assumption \eqref{eq:pastropgrand}. Since $\eta \leq 1/2$ on $I_0$, we also have
$$\tau_0 \partial_x \Psi \frac{(3 - 2 \eta) (\partial_x \eta)^2}{4(1 - \eta)^2} \geq 0,$$
for $x \in I_0$, and therefore,
\begin{equation}
\label{eq:laoupetit}
\tau_0 \partial_x \Psi \Big( (1 - 2 \eta) v^2 + \frac{\eta^2}{2} + X'(t_0) \eta v + \frac{(3 - 2 \eta) (\partial_x \eta)^2}{4(1 - \eta)^2} \Big) + \frac{\tau_0^3}{2} \partial_x^3 \Psi \big( \eta + \ln(1 - \eta) \big) \leq 0
\end{equation}
on $I_0$.

It remains to consider the case $x \notin I_0$. In view of the definition of $\eta_{\rm sup}$, we have
\begin{align*}
- \frac{2}{1 - \eta_{\rm sup}} \frac{(1 - \eta) v^2}{2} & \leq \quad (1 - 2 \eta) v^2 \quad \leq 4 \frac{(1 - \eta) v^2}{2},\\
0 & \leq \frac{(3 - 2 \eta) (\partial_x \eta)^2}{4(1 - \eta)^2} \leq \max \Big\{ 6, \frac{6}{1 - \eta_{\rm sup}} \Big\} \frac{(\partial_x \eta)^2}{8(1 - \eta)},
\end{align*}
as well as
$$\big| \eta + \ln(1 - \eta) \big| \leq \frac{2}{1 - \eta_{\rm sup}} \frac{\eta^2}{4}, \quad {\rm and} \quad |X'(t_0) \eta v| \leq 2 \sqrt{2} \frac{\eta^2}{4} + \frac{\sqrt{2}}{1 - \eta_{\rm sup}} \frac{(1 - \eta) v^2}{2},$$
where we have used the inequality $|X'(t_0)| \leq \sqrt{2}$. Therefore,
\begin{equation}
\label{eq:laougrand}
\begin{split}
\bigg| \tau_0 \partial_x \Psi \Big( (1 - 2 \eta) v^2 + \frac{\eta^2}{2} + X'(t_0) \eta v \Big) & + \frac{1}{2} \tau_0^3 \partial_x^3 \Psi \big( \eta + \ln(1 - \eta) \big) \bigg|\\
& \leq \tau_0 \Big( 6 + \frac{6}{1 - \eta_{\rm sup}} \Big) e(\eta, v) \sup_{|x| \geq \tau_0 R_0} \partial_x \Psi(x)
\end{split}
\end{equation}
on $\R \setminus I_0$, and the conclusion follows from \eqref{eq:laoupetit} and \eqref{eq:laougrand} by integration.
\end{proof}

We are finally in position to provide the

\begin{proof}[Proof of Proposition \ref{prop:mono}]
Let $t_0 \in [0, T]$. We apply Proposition \ref{prop:almomo} with the choices $t_0 := t_0$, $\Psi(x) := (1 + \th(x))/2$, $\tau_0 := \nu_{\gc^*}/16$ and $X(t) := (a_k(t) + a_{k + 1}(t))/2$. A quick computation first shows that for this choice of $\Psi$, the inequality $|\partial_x^3 \Psi| \leq C_0 \partial_x \Psi$ holds for $C_0 := 4$. The condition $2 C_0 \tau_0^2 < 1$ is satisfied since $\nu_{\gc^*}/16 < \sqrt{2}/16.$ We turn now to condition \eqref{eq:pastropgrand}. It follows from \eqref{def:Wtierce}, namely
$$\sqrt{2 - (a_k'(t_0))^2} \geq \frac{\nu_{\gc^*}}{2} \quad {\rm and} \quad \sqrt{2 - (a_{k + 1}'(t_0))^2} \geq \frac{\nu_{\gc^*}}{2},$$
that
\begin{equation}
\label{eq:ski0}
 X'(t_0)^2 \leq 2 - \frac{\nu_{\gc^*}^2}{4}.
\end{equation}
Using the inequality $1/(1 - s) \leq 1 + 2s$ for $s \in [0, 1/2]$, we infer from \eqref{eq:ski0} and the inequality $\nu_{\gc^*}/16 < 1/4$ that
$$\frac{1}{2} - \frac{X'(t_0)^2}{4 (1 - 8 (\nu^*)^2)} \geq \frac{1}{2} - \frac{2 - \frac{\nu_{\gc^*}^2}{4}}{4 (1 - 8 (\nu^*)^2)} \geq \frac{1}{2} - \Big( \frac{1}{2} - \frac{\nu_{\gc^*}^2}{16} \Big) \Big( 1 + 16 (\nu^*)^2 \Big) \geq \frac{\nu_{\gc^*}^2}{32}.$$
At this stage, we come back to the decomposition
$$\big( \eta(\cdot, t), v(\cdot, t) \big) = R_{\gc(t),\ga(t)}(\cdot) + \eps(\cdot, t).$$
In view of \eqref{est:stamodul1}, we have $\| \eps(\cdot, t) \|_X \leq K_1 \alpha$. Hence, by the Sobolev embedding theorem, we may choose $\alpha_4 \leq \alpha_3$ so that, if $\alpha \leq \alpha_4$, then 
\begin{equation}
\label{eq:ski1}
\| \eps(\cdot, t) \|_{L^\infty} \leq \frac{\nu_{\gc^*}^2}{64}.
\end{equation}
It follows from \eqref{eq:ski1} that \eqref{eq:pastropgrand} holds everywhere on the set 
$$S := \Big\{ x \in \R, \ {\rm s.t.} \ \big| \eta_{\gc(t_0), \ga(t_0)}(x) \big| \leq \frac{\nu_{\gc^*}^2}{64} \Big\}.$$
Combining the definition of $\eta_{\gc, \ga}$, the explicit form \eqref{form:etavc} of the solitons, condition \eqref{est:stamodul2bis} on $\ga(t_0)$ and the inequality $\nu_\gc \geq \nu_{\gc^*}/2$ stated in \eqref{eq:ouf1}, it follows that the set $S$ contains the interval $[X(t_0) - R_0, X(t_0) + R_0]$ where
$$R_0 := \frac{(L - 1) + \sigma^* t_0}{2} - \frac{2}{\nu_{\gc^*}} \ln \Big( \frac{256 N}{\nu_{\gc^*}^2} \Big).$$ 
Finally, since $0 < \partial_x \Psi(x) \leq 2 \exp(- 2 |x|)$ for any $x \in \R$, we obtain from \eqref{eq:almono} that
$$\frac{d}{dt} \big( Q_k(t) \big)_{|t = t_0} \leq \boO \Big( \exp \Big( - \frac{\nu_{\gc^*} (L + \sigma^* t_0)}{16} \Big) \Big),$$
and the proof is completed.
\end{proof}

We finally give the

\begin{proof}[Proof of Corollary \ref{cor:mono}]
Since the energy $E$ and the momentum $P$ are conserved quantities for equation \eqref{HGP}, we infer from definition \eqref{eq:decompG} that
$$\frac{d}{dt} \Big( G(t) \Big) = \sum_{k = 2}^N \big( c_k^* - c_{k + 1}^* \big)\frac{d}{dt} \Big( Q_k(t) \Big).$$
Corollary \ref{cor:mono} follows by combining Proposition \ref{prop:mono} and the ordering condition \eqref{eq:ordo}.
\end{proof}

\appendix
\section{A quantitative version of the implicit function theorem}
\label{sec:implicit}

In this appendix, we give the proof of the following version of the implicit function theorem which we have invoked in order to establish Proposition \ref{prop:modul}.

\begin{prop}
\label{prop:implicit}
Let $E$ and $F$ be two Banach spaces, $V$, an open subset of $F$, and $\phi$, a function of class $\boC^2$ from $E \times V$ to $F$. We assume that
\begin{itemize}
\item[(i)] there exist an open subset $U$ of $V$ and a map $\bar{x}$ of class $\boC^1$ from $U$ to $E$ such that
$$\phi(\bar{x}(y), y) = 0,$$
for any $y \in U$.
\item[(ii)] there exist a positive constant $K$, as well as a continuous isomorphism $A_y$ from $F$ to $F$, such that
\begin{equation}
\label{impl1}
d_y \phi_{(\bar{x}(y), y)} = A_y \big( Id + H_y \big),
\end{equation}
and
\begin{equation}
\label{impl2}
d_x \phi_{(\bar{x}(y), y)} = A_y T_y,
\end{equation}
for any $y \in U$. In \eqref{impl1}, the operator norm of the continuous linear mapping $H_y$ is less than $1/2$, whereas in \eqref{impl2}, the operator norm of $T_y$ is less than $K$.
\item[(iii)] the operator norm of $d^2 \phi_{(x, y)}$ is uniformly bounded with respect to $x \in E$ and $y \in V$.
\item[(iv)] there exists a positive constant $r$ such that
$$B(y, r) \subset V,$$
for any $y \in U$. 
\end{itemize}
Then, there exist two positive numbers $\delta$ and $\Lambda := 4 K + 4$, such that, given any $y \in U$, there exists a map $\pi_y \in \boC^1(B(\bar{x}(y), \delta), V)$ such that, given any $t \in B(\bar{x}(y), \delta)$, $z = \pi_y(t)$ is the unique solution in $B(y, \Lambda \delta)$ of the equation
$$\phi(t, z) = 0.$$
Moreover, the map $\pi_y$ is Lipschitz on $B(\bar{x}(y), \delta)$, with Lipschitz constant at most $\Lambda$.
\end{prop}

\begin{remark}
When $y$ is fixed in $U$, assumptions $(i)$ and $(ii)$ are enough to apply the implicit function theorem in order to construct a map $\pi_y \in \boC^1(B(\bar{x}(y), \delta), V)$ as in Proposition \ref{prop:implicit}. The difficulty comes from the fact that, in this case, the constants $\delta$ and $K$ depend a priori on $y$.
\end{remark}

\begin{proof}
The proof is similar to the usual proof of the implicit function theorem. We fix $y \in U$ and we set
$$\Phi(x, z) := \big( \Phi_1(x ,z), \Phi_2(x ,z) \big) := \big( x, A_y^{- 1} \phi(x, z) - T_y(x) \big).$$
The function $\Phi$ is well-defined and of class $\boC^2$ from $E \times V$ to $E \times F$. Moreover, its differential at the point $(\bar{x}(y), y)$ may be written as
$$d\Phi_{(\bar{x}(y), y)} = \begin{pmatrix} Id & 0 \\ 0 & Id + H(y) \end{pmatrix}.$$
In view of assumption $(ii)$, the operator norm of the maps $d\Phi_{(\bar{x}(y), y)} - Id$ is less than $1/2$ for any $y \in U$. Combining with assumption $(iii)$, we infer the existence of a positive number $\rho < r$ such that the operator norm $d\Phi_{(x, z)} - Id$ is less than $3/4$ for any $(x, z) \in B(\bar{x}(y), \rho) \times B(y, \rho)$.

Given a point $t \in E$, we next set
$$\Psi(x, z) := (t, - T_y(t)) + (x, z) - \Phi(x, z).$$
The map $\Psi$ is well-defined on $E \times V$, and Lipschitz on $B(\bar{x}(y), \rho) \times B(y, \rho)$, with Lipschitz constant at most $3/4$. Moreover, one can check that
$$\Psi(\bar{x}(y), y) = \big( t, y - T_y(t - \bar{x}(y)) \big),$$
so that, letting $\Lambda := 4 + 4 K$,
$$\big\| \Psi(\bar{x}(y), y) - (\bar{x}(y), y) \big\|_{E \times F} \leq \frac{\rho}{4},$$
when $t \in B(\bar{x}(y), \rho/\Lambda)$. In this case, the map $\Psi$ is a contraction from $B(\bar{x}(y), \rho) \times B(y, \rho)$ to $B(\bar{x}(y), \rho) \times B(y, \rho)$. By the fixed point theorem, there exists a unique point $(x_t, z_t)$ in $B(\bar{x}(y), \rho) \times B(y, \rho)$ such that $\Psi(x_t, z_t) = (x_t, z_t)$, i.e.
$$x_t = t \quad {\rm and} \quad \phi(t, z_t) = 0.$$

At this stage, we set $\delta := \rho/\Lambda < \rho < r$, and
$$\pi_y(t) := z_t \in V.$$
In view of the definition of $z_t$, given two points $t_1$ and $t_2$ in $B(\bar{x}(y), \delta)$, we have
$$\pi_y(t_1) - \pi_y(t_2) = - T_y(t_1 - t_2) + \big( \pi_y(t_1) - \Phi_2(t_1, \pi_y(t_1) - \pi_y(t_2) - \Phi_2(t_2, \pi_y(t_2) \big),$$
so that
$$\| \pi_y(t_1) - \pi_y(t_2) \|_F \leq K \| t_1 - t_2 \|_E + \frac{3}{4} \big( \| t_1 - t_2 \|_E + \| \pi_y(t_1) - \pi_y(t_2) \|_F \big).$$
As a consequence, the map $\pi_y$ is Lipschitz on $B(\bar{x}(y), \delta)$, with Lipschitz constant at most $4 K + 3 \leq \Lambda$. The fact that $\pi_y$ is of class $\boC^1$ on $B(\bar{x}(y), \delta)$ follows from the implicit function theorem. This completes the proof of Proposition \ref{prop:implicit}.
\end{proof}

\begin{merci}
P.G. and D.S. are partially sponsored by the project ``Around the dynamics of the Gross-Pitaevskii equation'' (JC09-437086) of the Agence Nationale de la Recherche. 
\end{merci}

\bibliographystyle{plain}
\bibliography{Bibliogr}

\begin{thebibliography}{10}

\bibitem{BetGrSa2}
F.~B\'ethuel, P.~Gravejat, and J.-C. Saut.
\newblock Existence and properties of travelling waves for the
  {Gross}-{Pitaevskii} equation.
\newblock In A.~Farina and J.-C. Saut, editors, {\em Stationary and time
  dependent {Gross}-{Pitaevskii} equations}, volume 473 of {\em Contemp.
  Math.}, pages 55--104. Amer. Math. Soc., Providence, RI, 2008.

\bibitem{BeGrSaS1}
F.~B\'ethuel, P.~Gravejat, J.-C. Saut, and D.~Smets.
\newblock Orbital stability of the black soliton for the {Gross}-{Pitaevskii}
  equation.
\newblock {\em Indiana Univ. Math. J}, 57(6):2611--2642, 2008.

\bibitem{BeGrSaS2}
F.~B\'ethuel, P.~Gravejat, J.-C. Saut, and D.~Smets.
\newblock On the {Korteweg}-de {Vries} long-wave approximation of the
  {Gross}-{Pitaevskii} equation {I}.
\newblock {\em Int. Math. Res. Not.}, 2009(14):2700--2748, 2009.

\bibitem{Chiron7}
D.~Chiron.
\newblock Travelling waves for the nonlinear {Schr\"odinger} equation with
  general nonlinearity in dimension one.
\newblock {\em Preprint}, 2011.

\bibitem{ChirRou2}
D.~Chiron and F.~Rousset.
\newblock The {KdV/KP-I} limit of the nonlinear {Schr\"odinger} equation.
\newblock {\em SIAM J. Math. Anal.}, 42(1):64--96, 2010.

\bibitem{DunfSch0}
N.~Dunford and J.T. Schwartz.
\newblock {\em Linear operators. {Part} {II}. {Spectral} theory. {Self}-adjoint
  operators in {Hilbert} space}, volume~7 of {\em Pure and Applied
  Mathematics}.
\newblock Interscience Publishers, John Wiley and Sons, New York-London-Sydney,
  1963.
\newblock With the assistance of W.G. Bade and R.G. Bartle.

\bibitem{FaddTak0}
L.D. Faddeev and L.A. Takhtajan.
\newblock {\em Hamiltonian methods in the theory of solitons}.
\newblock Classics in Mathematics. Springer-Verlag, Berlin-Heidelberg-New York,
  2007.
\newblock Translated by A.G. Reyman.

\bibitem{Gallo3}
C.~Gallo.
\newblock Schr\"odinger group on {Zhidkov spaces}.
\newblock {\em Adv. Differential Equations}, 9(5-6):509--538, 2004.

\bibitem{Gerard1}
P.~G\'erard.
\newblock The {Cauchy} problem for the {Gross}-{Pitaevskii} equation.
\newblock {\em Ann. Inst. Henri Poincar\'e, Analyse Non Lin\'eaire},
  23(5):765--779, 2006.

\bibitem{GeraZha1}
P.~G\'erard and Z.~Zhang.
\newblock Orbital stability of traveling waves for the one-dimensional
  {Gross}-{Pitaevskii} equation.
\newblock {\em J. Math. Pures Appl.}, 91(2):178--210, 2009.

\bibitem{GriShSt1}
M.~Grillakis, J.~Shatah, and W.A. Strauss.
\newblock Stability theory of solitary waves in the presence of symmetry {I}.
\newblock {\em J. Funct. Anal.}, 74(1):160--197, 1987.

\bibitem{LinZhiw1}
Z.~Lin.
\newblock Stability and instability of traveling solitonic bubbles.
\newblock {\em Adv. Differential Equations}, 7(8):897--918, 2002.

\bibitem{MartMer5}
Y.~Martel and F.~Merle.
\newblock Stability of two soliton collision for nonintegrable {gKdV}
  equations.
\newblock {\em Commun. Math. Phys.}, 286(1):39--79, 2009.

\bibitem{MartMer6}
Y.~Martel and F.~Merle.
\newblock Inelastic interaction of nearly equal solitons for the quartic {gKdV}
  equation.
\newblock {\em Invent. Math.}, 183(3):563--648, 2011.

\bibitem{MarMeTs1}
Y.~Martel, F.~Merle, and T.-P. Tsai.
\newblock Stability and asymptotic stability in the energy space of the sum of
  {$N$} solitons for subcritical {gKdV} equations.
\newblock {\em Commun. Math. Phys.}, 231(2):347--373, 2002.

\bibitem{MarMeTs2}
Y.~Martel, F.~Merle, and T.-P. Tsai.
\newblock Stability in {$H^1$} of the sum of {$K$} solitary waves for some
  nonlinear {Schr\"odinger} equations.
\newblock {\em Duke Math. J.}, 133(3):405--466, 2006.

\bibitem{Miura2}
R.M. Miura.
\newblock The {Korteweg}- de {Vries} equation: a survey of results.
\newblock {\em SIAM Rev.}, 18(3):412--459, 1976.

\bibitem{Tartous0}
H.M. Tartousi.
\newblock PhD thesis.
\newblock In preparation.

\bibitem{Vartani2}
A.H. Vartanian.
\newblock Long-time asymptotics of solutions to the {Cauchy} problem for the
  defocusing nonlinear {Schr\"odinger} equation with finite-density initial
  data. {II}. {Dark} solitons on continua.
\newblock {\em Math. Phys. Anal. Geom.}, 5(4):319--413, 2002.

\bibitem{ShabZak2}
V.E. Zakharov and A.B. Shabat.
\newblock Interaction between solitons in a stable medium.
\newblock {\em Sov. Phys. JETP}, 37:823--828, 1973.

\bibitem{Zhidkov1}
P.E. Zhidkov.
\newblock {\em {Korteweg}-{De} {Vries} and nonlinear {Schr\"odinger} equations
  : qualitative theory}, volume 1756 of {\em Lecture Notes in Mathematics}.
\newblock Springer-Verlag, Berlin, 2001.

\end{thebibliography}

\end{document}